\documentclass[12pt]{amsart}

\usepackage[latin1]{inputenc}
\usepackage[english]{babel}
\usepackage{indentfirst}
\usepackage{amssymb}
\usepackage{amsthm}
\usepackage{xcolor}
\usepackage[all]{xy}
\usepackage[mathscr]{eucal}

\newcommand{\impli}{\Rightarrow}
\newcommand{\Nat}{\mathbb{N}}
\newcommand{\N}{\mathbb{N}}
\newcommand{\erre}{\mathbb{R}}
\newcommand{\sub}{\subseteq}
\newcommand{\vf}{\varphi}

\newcommand{\cT}{\mathcal{T}}
\newcommand{\cA}{\mathcal{A}}
\def\epsilon{\varepsilon}

\newtheorem{theo}{Theorem}[section]
\newtheorem{lem}[theo]{Lemma}
\newtheorem{pro}[theo]{Proposition}
\newtheorem{cor}[theo]{Corollary}
\newtheorem{defi}[theo]{Definition}
\newtheorem{rem}[theo]{Remark}

\newtheorem{question}[theo]{Question}
\newtheorem{notation}[theo]{Notation}

\numberwithin{equation}{section}

\title{$\varepsilon$-weakly precompact sets in Banach spaces}

\author{Jos\'{e} Rodr\'{i}guez}
\address{Dpto. de Ingenier\'{i}a y Tecnolog\'{i}a de Computadores\\Facultad de Inform\'{a}tica\\
Universidad de Murcia\\ 30100 Espinardo (Murcia)\\ Spain} \email{joserr@um.es}

\subjclass[2010]{Primary: 46B50. Secondary: 46G10}

\keywords{$\varepsilon$-weakly precompact set; $\ell_1$-sequence; strongly weakly precompactly generated Banach space; Lebesgue-Bochner space}

\thanks{The research is partially supported by {\em Agencia Estatal de Investigaci\'{o}n} [MTM2017-86182-P, grant cofunded by ERDF, EU] 
and {\em Fundaci\'on S\'eneca} [20797/PI/18]}

\begin{document}

\begin{abstract}
A bounded subset $M$ of a Banach space~$X$ is said to be $\varepsilon$-weakly precompact, for a given $\epsilon\geq 0$, 
if every sequence $(x_n)_{n\in \N}$ in~$M$ admits a subsequence~$(x_{n_k})_{k\in \N}$ such that
$$
	\limsup_{k\to \infty}x^*(x_{n_k})-\liminf_{k\to\infty}x^*(x_{n_k}) \leq \epsilon
$$
for all $x^*\in B_{X^*}$. In this paper we discuss several aspects of $\epsilon$-weakly precompact sets.
On the one hand, we give quantitative versions of the following known results: (a)~the absolutely convex hull of 
a weakly precompact set is weakly precompact (Stegall), and (b)~for any probability measure~$\mu$, the set of all Bochner $\mu$-integrable functions taking values in a
weakly precompact subset of~$X$ is weakly precompact in $L_1(\mu,X)$ (Bourgain, Maurey, Pisier). On the other hand,   
we introduce a relative of a Banach space property considered by Kampoukos and Mercourakis when studying subspaces of strongly weakly compactly generated spaces. 
We say that a Banach space~$X$ has property~$\mathfrak{KM}_w$ if there is a family $\{M_{n,p}:n,p\in \Nat\}$ of subsets of~$X$ such that:
(i)~$M_{n,p}$ is $\frac{1}{p}$-weakly precompact for all $n,p\in \N$, and (ii)~for each weakly precompact set $C \sub X$ and 
for each $p\in \N$ there is $n\in \N$ such that $C \sub M_{n,p}$. All subspaces of strongly weakly precompactly generated spaces 
have property~$\mathfrak{KM}_w$. Among other things, we study the three-space problem and
the stability under unconditional sums of property~$\mathfrak{KM}_w$. 
\end{abstract}

\maketitle

\section{Introduction}

Let $X$ be a real Banach space. A set $M \sub X$ is said to be {\em weakly precompact} (or {\em conditionally weakly compact})
if every sequence in~$M$ admits a weakly Cauchy subsequence. Rosenthal's $\ell_1$-theorem~\cite{ros-J-3} states that a subset of a Banach space
is weakly precompact if (and only if) it is bounded and contains no $\ell_1$-sequence. Recall that a sequence~$(x_n)_{n\in \N}$ in~$X$
is said to be an {\em $\ell_1$-sequence} if it is equivalent to the usual basis of~$\ell_1$, i.e., it is bounded and there is a constant $c>0$ such that
\begin{equation*}
	\left\| \sum_{n=1}^m a_n x_n \right\|\geq c \sum_{n=1}^m |a_n|
\end{equation*}
for all $m\in \N$ and all $a_1,\dots,a_m\in \erre$. In this case, we will say that $(x_n)_{n\in \N}$ is an $\ell_1$-sequence {\em with constant~$c$}.
Behrends~\cite{beh-1} proved a quantitative version of Rosenthal's $\ell_1$-theorem which, loosely speaking, 
says that a bounded set containing $\ell_1$-sequences only with small constant  
must contain a sequence which is close to being weakly Cauchy. To state it properly we need to introduce some terminology.

Let $(x_n)_{n\in \N}$ be a bounded sequence in~$X$. 
We denote by
${\rm clust}_{X^{**}}(x_{n})$ the set of all $w^*$-cluster points of~$(x_n)_{n\in\N}$ in~$X^{**}$,
and we write $\delta(x_{n})$ to denote the diameter of ${\rm clust}_{X^{**}}(x_{n})$. It is easy to check that
\begin{eqnarray*}
	\delta(x_{n}) & = &\sup_{x^*\in B_{X^*}}
	\inf_{m\in \N}\sup_{n,n'\geq m}|x^*(x_n)-x^*(x_{n'})| \\ &=&
	\sup_{x^*\in B_{X^*}} \limsup_{n\to \infty}x^*(x_n)-\liminf_{n\to\infty}x^*(x_n).
\end{eqnarray*}
Note that $\delta(x_n)=0$ if and only if $(x_n)_{n\in \N}$ is weakly Cauchy. In a sense, $\delta(x_n)$ 
measures how far $(x_n)_{n\in \N}$ is from being weakly Cauchy. 
This quantitative approach to non-weakly Cauchy
sequences was considered in \cite{beh-1,gas-2,kac-alt,kal-alt,kal-spu-2,pfi-J}.
An elementary argument shows that {\em if $(x_n)_{n\in \N}$ is an $\ell_1$-sequence with constant~$c>0$, then 
$\delta(x_n)\geq 2c$} (see \cite[Lemma~5(i)]{kal-alt}).
Following \cite{beh-1}, we
say that $(x_n)_{n\in \N}$ {\em admits $\epsilon$-$\ell_1$-blocks}, for a given $\epsilon>0$, if  
for every infinite set $N \sub \N$ there exist a finite set $\{n_1,\dots,n_r\}\sub N$ and $a_1,\dots,a_r\in \erre$
with $\sum_{k=1}^r |a_k|=1$ such that
\begin{equation*}
	\left\| \sum_{k=1}^r a_k x_{n_k} \right\|\leq \epsilon.
\end{equation*}
Clearly, if $(x_n)_{n\in \N}$ admits no $\ell_1$-subsequence with constant~$\epsilon$, then 
$(x_n)_{n\in \N}$ admits $\epsilon$-$\ell_1$-blocks. Conversely, if $(x_n)_{n\in \N}$ 
admits $\epsilon$-$\ell_1$-blocks, then no subsequence of it can be an $\ell_1$-sequence with constant $c>\epsilon$. 
Behrends' theorem (see \cite[Theorem~3.2]{beh-1}) reads as follows:

\begin{theo}[Behrends]\label{theo:Behrends}
Let $X$ be a Banach space, $(x_n)_{n\in \N}$ be a bounded sequence in~$X$ and $\epsilon>0$.
If $(x_n)_{n\in \N}$ admits $\epsilon$-$\ell_1$-blocks, then 
there is a subsequence $(x_{n_k})_{k\in \N}$ such that $\delta(x_{n_k}) \leq 2\epsilon$.
\end{theo}

This leads to the following:

\begin{defi}\label{defi:epsilon-weaklyprecompact}
Let $X$ be a Banach space and $\epsilon\geq 0$. We say that a set $M \sub X$ is {\em $\epsilon$-weakly precompact}
if it is bounded and every sequence $(x_n)_{n\in \N}$ in~$M$ admits a subsequence~$(x_{n_k})_{k\in \N}$ such that
$\delta(x_{n_k})\leq \epsilon$.
\end{defi}

Clearly, a set is $0$-weakly precompact if and only if it is weakly precompact. 

In this paper we study $\epsilon$-weakly precompact sets and Banach spaces generated by them in a strong way
(see Definition~\ref{defi:KMw} below). In order to present our results we need further background.

A Banach space~$X$ is said to be {\em strongly weakly precompactly generated} ({\em SWPG} for short) if there is a weakly precompact set $C_0 \sub X$ with the
following property: for every weakly precompact set $C\sub X$ and for every $\epsilon>0$ there is $n\in \N$ such that $C \sub nC_0+\epsilon B_X$.
This class of spaces was introduced in~\cite{kun-sch} as a natural companion of the class of 
{\em strongly weakly compactly generated} ({\em SWCG} for short) spaces of Schl\"{u}chtermann and Wheeler~\cite{sch-whe}, which are those
satisfying the same condition but replacing weak precompactness with relative weak compactness. Neither of these two classes
is closed under subspaces. Indeed, Mercourakis and Stamati (see \cite[\S3]{mer-sta-2}) gave an example 
of a subspace of the SWCG space $L_1[0,1]$ which is not SWCG. The same example works for the property of being SWPG, 
because a Banach space is SWCG if and only if it is weakly sequentially complete and SWPG (see \cite[Theorem~2.5]{sch-whe} and \cite[Theorem~2.2]{laj-rod-2}).
In an attempt of characterizing subspaces of SWCG spaces, Kampoukos and Mercourakis~\cite{kam-mer2}
studied the following class of Banach spaces (called there spaces having ``property~(*)''):

\begin{defi}\label{defi:KM}
We say that a Banach space $X$ has {\em property~$\mathfrak{KM}$} if
there is a family $\{M_{n,p}:n,p\in \Nat\}$ of subsets of~$X$ such that:
\begin{enumerate}
\item[(i)] $M_{n,p}$ is $\frac{1}{p}$-relatively weakly compact for all $n,p\in \N$.
\item[(ii)] For each relatively weakly compact set $C \sub X$ and for each $p\in \N$ there is $n\in \N$ such that $C \sub M_{n,p}$.
\end{enumerate}
\end{defi}

A bounded set $M\sub X$ is said to be {\em $\epsilon$-relatively weakly compact}, for a given $\epsilon\geq 0$,
if $\overline{M}^{w^*} \sub X+\epsilon B_{X^{**}}$. Property~$\mathfrak{KM}$
stems somehow from a result proved in~\cite{fab-mon-ziz-2}
saying that a Banach space~$X$ is a subspace of a weakly compactly generated Banach space if and only if
there is a family $\{M_{n,p}:n,p\in \Nat\}$ of subsets of~$X$ satisfying condition~(i) of Definition~\ref{defi:KM} such that $X=\bigcup_{n\in \N}M_{n,p}$
for all $p\in \N$. Subspaces of SWCG spaces have property~$\mathfrak{KM}$ (see \cite[Proposition~2.15]{kam-mer2}),
but it is unknown whether the converse holds. 

In this paper we focus on the following related class of Banach spaces:

\begin{defi}\label{defi:KMw}
We say that a Banach space $X$ has {\em property~$\mathfrak{KM}_w$} if
there is a family $\{M_{n,p}:n,p\in \Nat\}$ of subsets of~$X$ such that:
\begin{enumerate}
\item[(i)] $M_{n,p}$ is $\frac{1}{p}$-weakly precompact for all $n,p\in \N$.
\item[(ii)] For each weakly precompact set $C \sub X$ and for each $p\in \N$ there is $n\in \N$ such that $C \sub M_{n,p}$.
\end{enumerate}
\end{defi}

The paper is organized as follows. In Section~\ref{section:epsilon-weaklyprecompact}
we discuss some aspects of $\epsilon$-weakly precompact sets in arbitrary Banach spaces.  On the one hand, we prove
that the absolutely convex hull of an $\epsilon$-weakly precompact set is $2\epsilon$-weakly precompact
(Theorem~\ref{theo:Krein}). This is a quantitative version of the well-known result
of Stegall (see \cite[Addendum]{ros-J-7}) that the absolutely convex hull of a weakly precompact set is weakly precompact.
On the other hand, we investigate $\epsilon$-weak precompactness in the Lebesgue-Bochner space $L_1(\mu,X)$,
where $X$ is a Banach space and $\mu$ is a probability measure. We show that 
if $(f_n)_{n\in \N}$ is a uniformly integrable sequence in~$L_1(\mu,X)$ such that the sequence $(f_n(\omega))_{n\in \N}$ is $\epsilon$-weakly precompact in~$X$
for $\mu$-a.e.~$\omega \in \Omega$, then $(f_n)_{n\in\N}$ is $2\epsilon$-weakly precompact in~$L_1(\mu,X)$ (Theorem~\ref{theo:Bourgain}). 
This is a quantitative version of a result due to Bourgain~\cite{bou-JJ}, Maurey and Pisier~\cite{pis3}.

In Section~\ref{section:propertyM} we study Banach spaces having property~$\mathfrak{KM}_w$. 
This class is closed under subspaces and includes all SWPG spaces (Proposition~\ref{pro:SWPGimpliesKMw})
as well as all spaces having property~$\mathfrak{KM}$ (Theorem~\ref{theo:wsc}). 
It is shown that a Banach space~$X$ has property~$\mathfrak{KM}_w$ whenever there is a subspace $Y \sub X$ not containing isomorphic copies of~$\ell_1$ such that
$X/Y$ has property~$\mathfrak{KM}_w$ (Theorem~\ref{theo:quotient-Mp}).
In Subsection~\ref{subsection:sums} we discuss the stability of this property under (countable) unconditional sums. While property~$\mathfrak{KM}_w$ is preserved
by $\ell_1$-sums (Proposition~\ref{pro:l1-sums}), in general this is not the case for $c_0$-sums or $\ell_p$-sums when $1<p<\infty$.
Indeed, if $E$ is any Banach space with a normalized $1$-unconditional basis and separable dual, and $(X_m)_{m\in \N}$ is a sequence of Banach spaces,
then the space $(\bigoplus_{m\in \N}X_m)_E$ fails property~$\mathfrak{KM}_{w}$ if $X_m$ contains
isomorphic copies of~$\ell_1$ for infinitely many $m\in \N$ (Theorem~\ref{theo:E-sequence}).
This extends some previous results on property~$\mathfrak{KM}$ and subspaces of SWPG spaces
obtained in~\cite{kam-mer2} and~\cite{laj-rod-2}. As an application, 
we show that the Banach space of Batt and Hiermeyer~\cite[\S3]{bat-hie}
fails property~$\mathfrak{KM}_w$ (Corollary~\ref{cor:embedding-l2l1}). Subsection~\ref{subsection:strongLB} 
contains some remarks on property~$\mathfrak{KM}_w$ within the setting of Lebesgue-Bochner spaces. Finally,
in Subsection~\ref{subsection:propertyK} we show that every Banach space having property~$\mathfrak{KM}$
satisfies the so-called property~(K) of Kwapie\'{n} introduced in~\cite{kal-pel} (see Theorem~\ref{theo:KMimpliesK}).

\subsubsection*{Notation and terminology}

Given a set $S$, its cardinality is denoted by $|S|$ and its power set is denoted by~$\mathcal{P}(S)$.
All our Banach spaces are real. By an {\em operator} we mean a continuous linear map between Banach spaces.
By a {\em subspace} of a Banach space we mean a norm closed linear subspace.
The topological dual of a Banach space~$X$ is denoted by~$X^*$ and
we write $w^*$ to denote the weak$^*$-topology on~$X^*$. 
The evaluation of $x^*\in X^*$ at $x\in X$ is denoted by either $x^*(x)$ or $\langle x,x^*\rangle$.
The norm of~$X$ is denoted by $\|\cdot\|_X$ or simply $\|\cdot\|$.
We write $B_X$ to denote the closed unit ball of~$X$, i.e., $B_{X}=\{x\in X:\|x\|\leq 1\}$. 
Given two sets $C_1,C_2 \sub X$, its Minkowski sum is $C_1+C_2:=\{x_1+x_2:\, x_1\in C_1, \, x_2\in C_2\}$.
Given a set $C\sub X$:
\begin{itemize} 
\item ${\rm co}(C)$ (resp., $\overline{{\rm co}}(C)$) denotes the convex (resp., closed convex) hull of~$C$,
\item ${\rm aco}(C)$ (resp., $\overline{{\rm aco}}(C)$) denotes the absolutely convex (resp., closed absolutely convex) hull of~$C$,
\item $\lambda C:=\{\lambda x: x\in C\}$ for every $\lambda \in \erre$, 
\item $\overline{{\rm span}}(C)$ denotes the subspace of~$X$ generated by~$C$.
\end{itemize}
Given a sequence $(x_n)_{n\in \N}$ in~$X$, by a {\em convex block subsequence} of~$(x_n)_{n\in \N}$   
we mean a sequence $(\tilde{x}_k)_{k\in \N}$ of the form
$$
	\tilde{x}_k=\sum_{n\in I_k}a_n x_n,
$$
where the $I_k$'s are finite subsets of~$\N$ with $\max(I_k) < \min(I_{k+1})$ and $(a_n)_{n\in \N}$ is a sequence
of non-negative real numbers such that $\sum_{n\in I_k}a_n=1$ for all~$k \in \N$.

\section{$\epsilon$-weakly precompact sets}\label{section:epsilon-weaklyprecompact}

For the sake of easy reference, the following statement gathers two fundamental
facts on $\epsilon$-weakly precompact sets and $\ell_1$-sequences which were mentioned in the introduction.  
The first one follows from Behrends' Theorem~\ref{theo:Behrends}, while the second one is elementary.
 
\begin{theo}\label{theo:Behrends-KPS}
Let $X$ be a Banach space, $M\sub X$ be a bounded set and $\epsilon>0$. 
\begin{enumerate}
\item[(i)] If $M$ is not is $\epsilon$-weakly precompact, then it contains an $\ell_1$-sequence with constant~$\frac{\epsilon}{2}$. 
\item[(ii)] If $M$ contains an $\ell_1$-sequence with constant~$\epsilon$, then $M$ cannot be $\epsilon'$-weakly precompact for any $0\leq \epsilon'<2\epsilon$.
\end{enumerate}
\end{theo}
\begin{proof}
(i) Let $(x_n)_{n\in \N}$ be a sequence in~$M$ such that $\delta(x_{n_k})>\epsilon$ for every subsequence $(x_{n_k})_{k\in \N}$. 
By Theorem~\ref{theo:Behrends}, $(x_n)_{n\in \N}$ does not admit $\frac{\epsilon}{2}$-$\ell_1$-blocks.
Therefore, $(x_n)_{n\in \N}$ admits a subsequence which is an $\ell_1$-sequence with constant~$\frac{\epsilon}{2}$.

(ii) Let $(x_n)_{n\in \N}$ be an $\ell_1$-sequence with constant~$\epsilon$ contained in~$M$. Every
subsequence $(x_{n_k})_{k\in \N}$ is also an $\ell_1$-sequence with constant~$\epsilon$
and therefore $\delta(x_{n_k})\geq 2\epsilon$ (see \cite[Lemma~5(i)]{kal-alt}). It follows
that $M$ cannot be $\epsilon'$-weakly precompact for any $0\leq \epsilon'<2\epsilon$.
\end{proof}

\begin{rem}\label{rem:ball}
If a Banach space~$X$ contains a subspace isomorphic to~$\ell_1$, then $B_X$ cannot be $\epsilon$-weakly precompact for any $0\leq \epsilon< 2$.
\end{rem}
\begin{proof} Fix $0<\eta < 1$. By James' $\ell_1$-distortion theorem (see, e.g., \cite[Theorem~10.3.1]{alb-kal}), there is a normalized
sequence in~$X$ which is an $\ell_1$-sequence with constant $1-\eta$. By Theorem~\ref{theo:Behrends-KPS}(ii), 
$B_X$ cannot be $\epsilon'$-weakly precompact for any $0\leq \epsilon'< 2(1-\eta)$.
\end{proof}

\subsection{Absolutely convex hulls}\label{subsection:aco}

It is known that the absolutely convex hull of a weakly precompact subset of a Banach space is weakly precompact. This was first
proved by Stegall, see \cite[Addendum]{ros-J-7} (cf. \cite[Corollary~1.1.9]{sch-PhD} and \cite[Corollary~B]{ros07}).
The purpose of this subsection is to give a quantitative version of that result, as follows:

\begin{theo}\label{theo:Krein}
Let $X$ be a Banach space, $M\sub X$ and $\epsilon\geq 0$. If
$M$ is $\epsilon$-weakly precompact, then ${\rm aco}(M)$ is $2\epsilon$-weakly precompact.  
\end{theo}

Our approach to Theorem~\ref{theo:Krein} will follow some ideas of Stegall's proof of the case $\epsilon=0$. 
We first need to introduce further terminology and to prove some auxiliary lemmata. 
Given a Banach space~$X$ and a bounded sequence $(x_n)_{n\in \N}$ in~$X$, we write
$$
	{\rm ca}(x_n):=\inf_{m\in \N}\sup_{n,n'\geq m}\|x_n-x_{n'}\|,
$$ 
which is a measure of how far $(x_n)_{n\in \N}$ is from being norm Cauchy. 
The following definition is related in a natural way to the classical measures of non-compactness of Hausdorff
and Kuratowski (see, e.g., \cite[\S 2.5]{kac-alt}):

\begin{defi}\label{defi:epsilon-precompact}
Let $X$ be a Banach space, $M \sub X$ and $\epsilon\geq 0$. We say that $M$
is {\em $\epsilon$-precompact} if it is bounded and every sequence $(x_n)_{n\in \N}$ in~$M$ admits a subsequence~$(x_{n_k})_{k\in \N}$
such that ${\rm ca}(x_{n_k}) \leq \epsilon$.
\end{defi}

\begin{lem}\label{lem:epsilon-compact+}
Let $X$ be a Banach space, $M \sub X$ and $\epsilon\geq 0$.
Then $M$ is $\epsilon$-precompact (resp., $\epsilon$-weakly precompact) if and only if it is 
$\epsilon'$-precompact (resp., $\epsilon'$-weakly precompact) for every $\epsilon'>\epsilon$.
\end{lem}
\begin{proof}
The ``only if'' part is obvious and the ``if'' part follows from a standard diagonalization argument.
\end{proof}

The following lemma is a quantitative version of Mazur's classical result that the absolutely convex hull of a relatively norm compact subset of a Banach space
is relatively norm compact (see, e.g., \cite[p.~51, Theorem~12]{die-uhl-J}).

\begin{lem}\label{lem:absconvex-precompact}
Let $X$ be a Banach space, $M \sub X$ and $\epsilon\geq 0$.
If $M$ is $\epsilon$-precompact, then ${\rm aco}(M)$ is $\epsilon$-precompact.
\end{lem}
\begin{proof}
Fix $\epsilon'>\epsilon$ and take $\eta>0$ small enough such that $\epsilon+2\eta \leq \epsilon'$. Observe that there is a finite set $F \sub X$ such that 
$$
	\sup_{x\in M}\min_{x'\in F}\|x-x'\| \leq \epsilon+\eta.
$$
Indeed, otherwise we could construct by induction a sequence $(x_n)_{n\in \N}$ in~$M$
such that $\|x_n-x_m\|>\epsilon+\eta$ whenever $n\neq m$, and therefore ${\rm ca}(x_{n_k})\geq \epsilon+\eta$
for every subsequence $(x_{n_k})_{k\in \N}$, which contradicts that $M$ is $\epsilon$-precompact.

Then $K:={\rm aco}(F)$ is norm compact and 
\begin{equation}\label{eqn:aco}
	\sup_{x\in {\rm aco}(M)}\min_{x'\in K}\|x-x'\|\leq \epsilon+\eta.
\end{equation}
Now let $(y_n)_{n\in \N}$ be a sequence in ${\rm aco}(M)$. By~\eqref{eqn:aco}, there is a sequence
$(z_n)_{n\in \N}$ in~$K$ such that $\|y_n-z_n\|\leq \epsilon+\eta$ for all $n\in \N$. Since $K$ is norm compact, $(z_n)_{n\in \N}$ admits
a subsequence $(z_{n_k})_{k\in \N}$ such that $\|z_{n_k}-z_{n_{k'}}\|\leq \eta$
and so $\|y_{n_k}-y_{n_{k'}}\|\leq \epsilon+2\eta\leq\epsilon'$ for all $k,k'\in \N$.
Hence ${\rm ca}(y_{n_k})\leq \epsilon'$. This shows that ${\rm aco}(M)$ is $\epsilon'$-precompact.
As $\epsilon'>\epsilon$ is arbitrary, ${\rm aco}(M)$ is $\epsilon$-precompact (Lemma~\ref{lem:epsilon-compact+}).
\end{proof}

We will also need the following simple lemma.

\begin{lem}\label{lem:operator-compact}
Let $T:X\to Y$ be an operator between the Banach spaces $X$ and~$Y$, $M \sub X$ and $\epsilon\geq 0$. If $M$
is $\epsilon$-precompact (resp., $\epsilon$-weakly precompact), then $T(M)$ is $\|T\|\epsilon$-precompact
(resp., $\|T\|\epsilon$-weakly precompact). 
\end{lem}
\begin{proof}
Just bear in mind that, for any bounded sequence $(x_n)_{n\in \N}$ in~$X$, we have 
${\rm ca}(T(x_n)) \leq \|T\|{\rm ca}(x_n)$ and $\delta(T(x_n)) \leq \|T\|\delta(x_n)$.
\end{proof}

Given a compact Hausdorff topological space~$K$ and a regular Borel probability measure $\mu$ on~$K$,
the operator $i_\mu:C(K) \to L_1(\mu)$ that sends each function to its equivalence class is completely continuous (i.e.,
it maps weakly Cauchy sequences to norm convergent sequences), as an immediate consequence of Lebesgue's dominated convergence theorem.
We next provide a quantitative version of this fact.

\begin{lem}\label{lem:inclusionoperator}
Let $K$ be a compact Hausdorff topological space, $\mu$ be a regular Borel probability measure on~$K$, and 
$\epsilon\geq 0$. If $M \sub C(K)$ is $\epsilon$-weakly precompact, then 
$i_\mu(M)$ is $\epsilon$-precompact in~$L_1(\mu)$.
\end{lem}
\begin{proof}
By Lemma~\ref{lem:epsilon-compact+}, it suffices to
check that $i_\mu(M)$ is $\epsilon'$-precompact for every $\epsilon'>\epsilon$.
Write $\alpha:=\sup\{\|f\|_{C(K)}: f \in M\}$ and choose $\eta>0$ small enough such that 
$2\alpha\eta+\epsilon\leq \epsilon'$. Let $(f_n)_{n\in \N}$ be a sequence in~$M$.
Since $M$ is $\epsilon$-weakly precompact, by passing to a subsequence we can assume that $\delta(f_n)\leq \epsilon$.
By \cite[Proposition~9.1]{kac-alt}, there is a closed set $L \sub K$ such that $\mu(K\setminus L)\leq \eta$
and the sequence of restrictions $(f_n|_L)_{n\in \N}$ in~$C(L)$ satisfies
${\rm ca}(f_n|_{L}) \leq \delta(f_n)$, hence ${\rm ca}(f_n|_{L}) \leq \epsilon$. 
Observe that for each $k,k'\in \N$ we have
\begin{multline*}
	\|i_\mu(f_{k})-i_\mu(f_{k'})\|_{L_1(\mu)} = \int_K|f_{k}-f_{k'}|\, d\mu \\ =
	\int_{K\setminus L}|f_{k}-f_{k'}|\, d\mu+\int_L|f_{k}-f_{k'}|\, d\mu  
	 \leq  2\alpha \eta + \big\|f_{k}|_L-f_{k'}|_L\big\|_{C(L)}, 
\end{multline*}
hence 
$$
	{\rm ca}(i_\mu(f_n)) \leq 2\alpha\eta +{\rm ca}(f_n|_{L}) \leq 2\alpha\eta+\epsilon \leq \epsilon'.
$$
This shows that $i_\mu(M)$ is $\epsilon'$-precompact for every $\epsilon'> \epsilon$.
\end{proof}

\begin{lem}\label{lem:inclusionoperatorLinfinity}
Let $(\Omega,\Sigma,\mu)$ be a probability space, $j_\mu:L_\infty(\mu) \to L_1(\mu)$ be the inclusion operator, and $\epsilon\geq 0$.
If $M \sub L_\infty(\mu)$ is $\epsilon$-weakly precompact, then 
$j_\mu(M)$ is $\epsilon$-precompact in~$L_1(\mu)$.
\end{lem}
\begin{proof}
Let $K$ be the Stone space of the measure algebra of~$\mu$ and let $\tilde{\mu}$ be the regular Borel probability on~$K$
induced by~$\mu$. Then there exist isometric isomorphisms $I_\infty: L_\infty(\mu) \to C(K)$ and $I_1: L_1(\tilde{\mu})\to L_1(\mu)$
such that the diagram
$$
	\xymatrix@R=3pc@C=3pc{L_\infty(\mu)
	\ar[r]^{j_\mu} \ar[d]_{I_\infty} & L_1(\mu)\\
	C(K)  \ar[r]^{i_{\tilde{\mu}}}  & L_1(\tilde{\mu}) \ar[u]_{I_1} \\
	}
$$
commutes.
Since $M$ is $\epsilon$-weakly precompact and $\|I_\infty\|=1$, the set 
$I_\infty(M)$ is $\epsilon$-weakly precompact (Lemma~\ref{lem:operator-compact}). Now, we can apply Lemma~\ref{lem:inclusionoperator} to
deduce that $i_{\tilde{\mu}}(I_\infty(M))$ is $\epsilon$-precompact. Since $\|I_1\|=1$, we conclude
that $j_\mu(M)$ is $\epsilon$-precompact (Lemma~\ref{lem:operator-compact} again).
\end{proof}

We are now ready to prove Theorem~\ref{theo:Krein}.

\begin{proof}[Proof of Theorem~\ref{theo:Krein}]
By Lemma~\ref{lem:epsilon-compact+}, it suffices to
check that ${\rm aco}(M)$ is $\epsilon'$-weakly precompact for every $\epsilon'>2\epsilon$.
To this end we will apply Theorem~\ref{theo:Behrends-KPS}(i), that is, we will show that if $(x_n)_{n\in \N}$ is 
an $\ell_1$-sequence with constant~$C>0$ contained in~${\rm aco}(M)$, then $C<\frac{\epsilon'}{2}$.

Let  $(r_n)_{n\in\N}$ be the sequence of Rademacher functions on~$[0,1]$ and let 
$$
	T_0:\overline{{\rm span}}(\{x_n:n\in \N\}) \to L_\infty[0,1]
$$ 
be the unique operator satisfying $T_0(x_n)=r_n$ for all $n\in \N$, so that $\|T_0\|\leq \frac{1}{C}$. 
Since $L_\infty[0,1]$ is isometrically injective (see, e.g., \cite[Proposition~4.3.8(ii)]{alb-kal}), 
$T_0$ extends to an operator $\tilde{T}_0:X \to L_\infty[0,1]$ such that $\|\tilde{T}_0\|=\|T_0\|$.
Define 
$$
	T:=j \circ \tilde{T}_0: X \to L_1[0,1],
$$
where $j:L_\infty[0,1]\to L_1[0,1]$ is the inclusion operator. 

On the one hand, $r_n=T(x_n) \in T({\rm aco}(M))={\rm aco}(T(M))$ for all $n\in \N$
and we have $\|r_n-r_{n'}\|_{L_1[0,1]}=1$ whenever $n\neq n'$.
Therefore, ${\rm aco}(T(M))$ cannot be $\epsilon''$-precompact for any $0\leq \epsilon''< 1$.

On the other hand, since $M$ is $\epsilon$-weakly precompact and $\|\tilde{T}_0\|\leq \frac{1}{C}$, the set $\tilde{T}_0(M)$ is $\frac{\epsilon}{C}$-weakly precompact 
in~$L_\infty[0,1]$ (Lemma~\ref{lem:operator-compact}). Now, Lemma~\ref{lem:inclusionoperatorLinfinity} ensures that $T(M)$ is $\frac{\epsilon}{C}$-precompact
in $L_1[0,1]$ and therefore ${\rm aco}(T(M))$ is $\frac{\epsilon}{C}$-precompact as well (Lemma~\ref{lem:absconvex-precompact}).
It follows that $1\leq \frac{\epsilon}{C}$ and so $C < \frac{\epsilon'}{2}$, as we wanted.
\end{proof}

\begin{question}\label{question:aco}
Is constant $2$ optimal in Theorem~\ref{theo:Krein}?
\end{question}

\subsection{Lebesgue-Bochner spaces}\label{subsection:LB}

Let $(\Omega,\Sigma,\mu)$ be a probability space and $X$ be a Banach space.
The characteristic function of any $A\in \Sigma$ is denoted by~$\chi_A$. 
Given a function $f:\Omega \to X$, we denote by $\|f(\cdot)\|_X$ the real-valued function on~$\Omega$
defined by $\omega \mapsto \|f(\omega)\|_X$. As usual, $L_\infty(\mu,X)$ is the Banach space of all
(equivalence classes of) strongly $\mu$-measurable functions $f:\Omega \to X$ which are $\mu$-essentially bounded, equipped with the norm 
$$
	\|f\|_{L_\infty(\mu,X)}:=\big\| \|f(\cdot)\|_X \big\|_{L_\infty(\mu)}. 
$$
We denote by $L_1(\mu,X)$ the Banach space of all (equivalence classes of) 
Bochner $\mu$-integrable functions $f:\Omega \to X$, equipped with the norm 
$$
	\|f\|_{L_1(\mu,X)}:=\int_\Omega \|f(\cdot)\|_X \, d\mu.
$$
A set $W\sub L_1(\mu,X)$ is said to be {\em uniformly integrable} if it is bounded
and for every $\epsilon>0$ there is $\delta>0$ such that $\sup_{f\in W}\int_A \|f(\cdot)\|_X \, d\mu\leq \epsilon$
for every $A\in \Sigma$ with $\mu(A)\leq\delta$. The simplest example of a uniformly integrable set is 
$$
	L(M):=\{f\in L_1(\mu,X): \, f(\omega)\in M \text{ for $\mu$-a.e. $\omega \in \Omega$}\} 
$$
for a bounded set $M \sub X$. Every weakly precompact subset of $L_1(\mu,X)$ is
uniformly integrable (see, e.g., \cite[p.~104, Theorem~4]{die-uhl-J}), but the converse does not hold in general.
The most penetrating study of weak precompactness in Lebesgue-Bochner spaces was
made by Talagrand~\cite{tal11}. Here we will focus on an earlier result proved independently by Bourgain (see \cite[Theorem~8]{bou-JJ}), Maurey and Pisier~\cite{pis3}:  
{\em if $(f_n)_{n\in \N}$ is a uniformly integrable sequence in~$L_1(\mu,X)$ such that the sequence
$(f_n(\omega))_{n\in \N}$ is weakly precompact in~$X$
for $\mu$-a.e. $\omega \in \Omega$, then $(f_n)_{n\in\N}$ is weakly precompact in~$L_1(\mu,X)$ (cf. \cite[Corollary~10]{tal11}).}
The purpose of this subsection is to prove the following quantitative version of the Bourgain-Maurey-Pisier theorem:

\begin{theo}\label{theo:Bourgain}
Let $(\Omega,\Sigma,\mu)$ be a probability space, $X$ be a Banach space, and $\epsilon\geq 0$.
Let $F:\Omega \to \mathcal{P}(X)$ be a multi-function such that $F(\omega)$ is $\epsilon$-weakly precompact
for $\mu$-a.e. $\omega \in \Omega$. Write
$$
	S_{1}(F):=\{f\in L_1(\mu,X): \, f(\omega)\in F(\omega) \text{ for $\mu$-a.e. $\omega\in \Omega$}\}
$$
to denote the set of all (equivalence classes of) Bochner $\mu$-integrable selectors of~$F$. Then:
\begin{enumerate}
\item[(i)] If $(f_n)_{n\in \N}$ is a uniformly integrable sequence in
$S_{1}(F)$, then $(f_n)_{n\in\N}$ cannot be an $\ell_1$-sequence with constant~$C>\epsilon$.
\item[(ii)] Every uniformly integrable subset of $S_{1}(F)$ is $2\epsilon$-weakly precompact in $L_1(\mu,X)$. 
\end{enumerate}
\end{theo}

When applied to constant multi-functions, the previous theorem yields:

\begin{cor}\label{cor:Bourgain}
Let $(\Omega,\Sigma,\mu)$ be a probability space, $X$ be a Banach space, and $\epsilon\geq 0$. If $M \sub X$ is $\epsilon$-weakly precompact, then 
$L(M)$ is $2\epsilon$-weakly precompact in~$L_1(\mu,X)$.
\end{cor}

To prove Theorem~\ref{theo:Bourgain} we will follow 
the approach to the Bourgain-Maurey-Pisier theorem which can be found in \cite[\S2.2]{cem-men}. 
We need some previous lemmata. The first one is a straightforward application of Fatou's lemma.

\begin{lem}\label{lem:Fatou}
Let $(\Omega,\Sigma,\mu)$ be a probability space and $(g_n)_{n\in \N}$ be a bounded sequence in~$L_\infty(\mu)$.
Then 
$$
	\int_\Omega \liminf_{n\to \infty}g_n\, d\mu \leq \liminf_{n\to \infty} \int_\Omega g_n\, d\mu
	\leq \limsup_{n\to \infty} \int_\Omega g_n\, d\mu \leq
	\int_\Omega \limsup_{n\to \infty}g_n\, d\mu.
$$
\end{lem}
\begin{proof}
Write $C:=\sup_{n\in \N}\|g_n\|_{L_\infty(\mu)}<\infty$.
Now, we can apply Fatou's lemma to the sequences $(C+g_n)_{n\in \N}$ and $(C-g_n)_{n\in \N}$ to get 
the desired inequalities.
\end{proof}

\begin{lem}\label{lem:limsup-liminf}
Let $(\Omega,\Sigma,\mu)$ be a probability space, $(h_n)_{n\in \N}$ be a bounded sequence in~$L_\infty(\mu)$, and 
$\epsilon\geq 0$. If 
$$
	\limsup_{n\to \infty}h_n-\liminf_{n\to \infty}h_n \leq \epsilon\quad \text{$\mu$-a.e.,} 
$$
then
$$
	\limsup_{m\to \infty}\int_\Omega \left|h_m-\liminf_{n\to \infty}h_n\right| \, d\mu \leq \epsilon.
$$
\end{lem}
\begin{proof}
Write 
$$
	g_m:=\left|h_m-\liminf_{n\to \infty}h_n\right|
	\quad\text{for all $m\in \N$.}
$$
Then $(g_m)_{m\in \N}$ is a bounded sequence in~$L_\infty(\mu)$ with $\liminf_{m\to\infty}g_m=0$. We have
\begin{multline*}
	\limsup_{m\to \infty} g_m=
	\limsup_{m\to \infty} g_m - \liminf_{m\to \infty} g_m  = 
	\inf_{n\in \N}\sup_{m,m'\geq n}|g_m-g_{m'}|
	\\  \leq 
	\inf_{n\in \N}\sup_{m,m'\geq n}|h_m-h_{m'}|
	=
	\limsup_{m\to \infty} h_m - \liminf_{m\to \infty} h_{m} \leq \epsilon
	\quad\text{$\mu$-a.e.}
\end{multline*}
Therefore, from Lemma~\ref{lem:Fatou} it follows that
$$
	\limsup_{m\to \infty}\int_\Omega g_m \, d\mu \leq \int_\Omega \limsup_{m\to \infty} g_m \, d\mu \leq \epsilon,
$$
as required.
\end{proof}

Let $(\Omega,\Sigma,\mu)$ be a probability space and $X$ be a Banach space. 
Recall that a function $\varphi:\Omega \to X^*$ is said to be {\em $w^*$-scalarly $\mu$-measurable}
if for every $x\in X$ the composition $\langle x,\varphi(\cdot) \rangle:\Omega \to \mathbb R$ is $\mu$-measurable.
It is known that any element of $L_1(\mu,X)^*$ can be identified with a $w^*$-scalarly $\mu$-measurable function 
$\varphi: \Omega \to X^*$ in such a way that $\|\varphi(\cdot)\|_{X^*} \in L_\infty(\mu)$ and 
$\|\varphi\|_{L_1(\mu,X)^*}=\|\|\varphi(\cdot)\|_{X^*}\|_{L_\infty(\mu)}$,
the duality being
$$
	\langle h,\varphi \rangle=\int_\Omega \langle h(\cdot),\varphi(\cdot)\rangle \, d\mu
	\quad
	\text{for all $h\in L_1(\mu,X)$}
$$
(see, e.g., \cite[Theorem~1.5.4]{cem-men}). We will use this representation
of $L_1(\mu,X)^*$ in the proofs of Lemmas~\ref{lem:BH} and~\ref{lem:simple} below.

\begin{notation}\label{notation:tensor}
\rm Given $x\in X$ and $f\in L_1(\mu)$, we write $f \otimes x \in L_1(\mu,X)$ to 
denote the (equivalence class of the) function defined by 
$$
	(f\otimes x)(\omega):=f(\omega)x
	\quad \text{for $\mu$-a.e. $\omega \in \Omega$.}
$$
Observe that $\|f\otimes x\|_{L_1(\mu,X)} = \|f\|_{L_1(\mu)}\|x\|_X$.
\end{notation}

Throughout the rest of this subsection the unit interval $[0,1]$ is equipped with the Lebesgue measure
and we denote by $(r_n)_{n\in \N}$ the sequence of Rademacher functions on~$[0,1]$. 

\begin{lem}\label{lem:BH-0}
Let $Z$ be a Banach space and $(z_n)_{n\in \N}$ be an $\ell_1$-sequence in~$Z$ with constant~$C>0$. 
Then $(r_n\otimes z_n)_{n\in \N}$ is an $\ell_1$-sequence in~$L_1([0,1],Z)$ with constant~$C$
and so
$$
	{\rm clust}_{L_1([0,1],Z)^{**}}(r_n\otimes z_n) \not\subseteq \epsilon B_{L_1([0,1],Z)^{**}} 
$$ 
for any $0\leq \epsilon<C$.
\end{lem}
\begin{proof}
The first statement follows from a simple computation 
(see, e.g., the proof of Proposition~2.2.1 in~\cite{cem-men}). We have $\delta(r_n\otimes z_n) \geq 2C$
by~\cite[Lemma~5(i)]{kal-alt}, which clearly implies the second statement.
\end{proof}

\begin{lem}\label{lem:BH}
Let $Z$ be a Banach space, $M \sub Z$ and $\epsilon\geq 0$. 
If $M$ is $\epsilon$-weakly precompact, then
for every sequence $(z_n)_{n\in \N}$ in~$M$ we have
$$
	{\rm clust}_{L_1([0,1],Z)^{**}}(r_n\otimes z_n) \sub \epsilon B_{L_1([0,1],Z)^{**}}. 
$$
\end{lem}
\begin{proof} Note that $M$ is bounded and so $(r_n\otimes z_n)_{n\in \N}$ is bounded.
Fix an arbitrary $F\in {\rm clust}_{L_1([0,1],Z)^{**}}(r_{n}\otimes z_{n})$. We claim that
$$
	\big|\langle F, \varphi \rangle \big| \leq \epsilon
	\quad
	\text{for every $\varphi\in B_{L_1([0,1],Z)^{*}}$.}
$$ 
Indeed, take any $\varphi\in B_{L_1([0,1],Z)^{*}}$ (represented as in the paragraph preceding Notation~\ref{notation:tensor}).
Let $(r_{n_k}\otimes z_{n_k})_{k\in \N}$ be a subsequence such that 
$$
	\langle r_{n_k}\otimes z_{n_k},\varphi \rangle \to \langle F,\varphi\rangle
	\quad\text{as $k\to \infty$.}
$$ 
Since $M$ is $\epsilon$-weakly precompact, by passing to a further subsequence
we can assume that $\delta(z_{n_k})\leq \epsilon$. 
Write $h:=\liminf_{k\to \infty} \langle z_{n_k},\varphi(\cdot)\rangle \in L_\infty[0,1]$. Since
$$
	\limsup_{k\to \infty} \langle z_{n_k},\varphi(t)\rangle - \liminf_{k\to \infty} \langle z_{n_k},\varphi(t)\rangle
	\leq \delta(z_{n_k}) \leq \epsilon
	\quad\text{for a.e. $t\in [0,1]$},
$$
we can apply Lemma~\ref{lem:limsup-liminf} to get
\begin{equation}\label{eqn:limsup-integral}
	\limsup_{k\to \infty} \int_0^1 \big|\langle z_{n_k},\varphi(t)\rangle-h(t) \big| \, dt \leq \epsilon.
\end{equation}

For each $k\in \N$ we have
$$
	\langle r_{n_k}\otimes z_{n_k},\varphi \rangle=\int_0^1 r_{n_k}(t)\langle z_{n_k},\varphi(t)\rangle \, dt
$$
and so
\begin{eqnarray*}
	\big| \langle r_{n_k}\otimes z_{n_k},\varphi \rangle \big|
	 & = &
	\left|
	\int_0^1 r_{n_k}(t) h(t) \, dt + \int_0^1 r_{n_k}(t)\big( \langle z_{n_k},\varphi(t) \rangle - h(t) \big) \, dt
	\right|
	\\  & \leq &
	\left|\int_0^1 r_{n_k}(t) h(t) \, dt\right|+
	\int_0^1 \big|\langle z_{n_k},\varphi(t)\rangle-h(t) \big| \, dt.
\end{eqnarray*}
This inequality, \eqref{eqn:limsup-integral} and the fact that $(r_{n_k})_{k\in \N}$ is weakly null in $L_1[0,1]$ yield
$| \langle F, \varphi\rangle|\leq \epsilon$. The proof is finished.
\end{proof}

\begin{lem}\label{lem:simple}
Let $(\Omega,\Sigma,\mu)$ be a probability space, $Y$ be a Banach space, and $\epsilon\geq 0$. Let $(h_n)_{n\in \N}$ be a sequence in $L_\infty(\mu,Y)$ 
such that 
$$
	\sup_{n\in \N}\|h_n\|_{L_\infty(\mu,Y)}<\infty
$$ 
and 
$$
		{\rm clust}_{Y^{**}}(h_n(\omega)) \sub \epsilon B_{Y^{**}} 
		\quad\text{for $\mu$-a.e. $\omega \in \Omega$.}
$$ 
Then 
$$
	{\rm clust}_{L_1(\mu,Y)^{**}}(h_n) \sub \epsilon B_{L_1(\mu,Y)^{**}}.
$$
\end{lem}
\begin{proof} Note that $(h_n)_{n\in \N}$ is bounded in $L_1(\mu,X)$.
Fix $H\in {\rm clust}_{L_1(\mu,Y)^{**}}(h_n)$. Take any $\varphi\in B_{L_1(\mu,Y)^{*}}$
(represented as in the paragraph preceding Notation~\ref{notation:tensor}). 
For each $n\in \N$ we define $g_n\in L_\infty(\mu)$ by $g_n:=|\langle h_n(\cdot),\varphi(\cdot)\rangle|$.
By the assumptions, we have $\sup_{n\in \N}\|g_n\|_{L_\infty(\mu)} <\infty$ and 
$\limsup_{n\to \infty} g_n \leq \epsilon$ $\mu$-a.e. Now, we can apply Lemma~\ref{lem:Fatou} to get
$$
	\limsup_{n\to\infty} \int_\Omega g_n \, d\mu
	\leq
	\int_\Omega \limsup_{n\to\infty} g_n \, d\mu \leq \epsilon,
$$
and so
\begin{eqnarray*}
	\big| \langle H,\varphi \rangle\big|
	 & \leq &
	\limsup_{n\to\infty} \big| \langle h_n,\varphi \rangle\big|=\limsup_{n\to \infty} \left|\int_{\Omega} \langle h_n(\cdot),\varphi(\cdot)\rangle \, d\mu \right| \\
	& \leq &
	 \limsup_{n\to\infty} \int_\Omega g_n \, d\mu \leq \epsilon.
\end{eqnarray*}
This shows that ${\rm clust}_{L_1(\mu,Y)^{**}}(h_n) \sub \epsilon B_{L_1(\mu,Y)^{**}}$.
\end{proof}

The following lemma belongs to the folklore. We include a proof since we did not find any suitable reference for it.

\begin{lem}\label{lem:double}
Let $(\Omega,\Sigma,\mu)$ be a probability space, $X$ be a Banach space, $f:\Omega \to X$
be a strongly $\mu$-measurable function and $g\in L_1[0,1]$. Define $h_{g,f}: \Omega \to L_1([0,1],X)$ by $h_{g,f}(\omega):=g \otimes f(\omega)$
for all $\omega \in \Omega$. Then: 
\begin{enumerate}
\item[(i)] $h_{g,f}$ is strongly $\mu$-measurable; 
\item[(ii)] $h_{g,f}$ is $\mu$-essentially bounded whenever $f$ is $\mu$-essentially bounded.
\end{enumerate}
\end{lem}
\begin{proof} (i) Clearly, $h_{g,f}$ is a simple function whenever $f$ is. In the general case, if $f_n:\Omega \to X$
is a sequence of simple functions converging to~$f$ $\mu$-a.e., then 
$(h_{g,f_n})_{n\in \N}$ is a sequence of simple functions converging to~$h_{g,f}$ $\mu$-a.e., because
\begin{multline*}
	\|h_{g,f_n}(\omega)-h_{g,f}(\omega)\|_{L_1([0,1],X)} \\ =
	\|g\otimes(f_n(\omega)-f(\omega))\|_{L_1([0,1],X)}=
	\|g\|_{L_1[0,1]} \|f_n(\omega)-f(\omega)\|_X 
\end{multline*}
for every $\omega\in \Omega$ and for every $n\in \N$. Thus, $h_{g,f}$ is strongly $\mu$-measurable.

(ii) This is immediate from~(i) and the equality 
$$
	\|h_{g,f}(\omega)\|_{L_1([0,1],X)}=\|g\|_{L_1[0,1]}\|f(\omega)\|_{X}
$$
which holds for every $\omega \in \Omega$.
\end{proof}

We also isolate for easy reference the following standard fact, which follows
from Chebyshev's inequality. 

\begin{lem}\label{lem:split}
Let $(\Omega,\Sigma,\mu)$ be a probability space and $X$ be a Banach space. If $W\sub L_1(\mu,X)$ is uniformly integrable,
then for every $\eta>0$ there is $\rho>0$ such that 
$$
	W \sub L(\rho B_X) + \eta B_{L_1(\mu,X)}.
$$ 
More precisely, for every $f\in W$ there is $A\in \Sigma$ such that $f\chi_A\in L(\rho B_X)$ and $\|f\chi_{\Omega \setminus A}\|_{L_1(\mu,X)}\leq \eta$.
\end{lem}

We have already gathered all the tools needed to prove Theorem~\ref{theo:Bourgain}.

\begin{proof}[Proof of Theorem~\ref{theo:Bourgain}] 
(ii) follows from (i), Theorem~\ref{theo:Behrends-KPS}(i) and Lemma~\ref{lem:epsilon-compact+}.

We begin the proof of (i) with the following:

{\em Particular case. Suppose that $f_n\in L_\infty(\mu,X)$ for all $n\in \N$ and that}
$$
	\sup_{n\in \N} \|f_n\|_{L_\infty(\mu,X)}<\infty.
$$
Write $Y:=L_1([0,1],X)$. For each $n\in \N$ we define (thanks to Lemma~\ref{lem:double})
$h_n \in L_\infty(\mu,Y)$ by 
$$
	h_n(\omega):=r_n\otimes f_n(\omega) 
	\quad\text{for $\mu$-a.e. $\omega\in \Omega$.}
$$
By the assumption on~$F$ and Lemma~\ref{lem:BH}, we have  
$$
	{\rm clust}_{Y^{**}}(h_n(\omega)) \sub \epsilon B_{Y^{**}}
	\quad\text{for $\mu$-a.e. $\omega\in \Omega$.} 
$$
Bearing in mind that
$\sup_{n\in \N}\|h_n\|_{L_\infty(\mu,Y)}<\infty$,
an appeal to Lemma~\ref{lem:simple} ensures that
$$
	{\rm clust}_{L_1(\mu,Y)^{**}}(h_n) \sub \epsilon B_{L_1(\mu,Y)^{**}}.
$$

Write $Z:=L_1(\mu,X)$ and let $\Phi: L_1(\mu,Y) \to L_1([0,1],Z)$ be the natural isometric isomorphism. Then
$\Phi(h_n)=r_n\otimes f_n$ for all $n\in \N$ and therefore
$$
	{\rm clust}_{L_1([0,1],Z)^{**}}(r_n\otimes f_n) \sub \epsilon B_{L_1([0,1],Z)^{**}}.
$$
Therefore, $(f_n)_{n\in\N}$ cannot be an $\ell_1$-sequence with constant $C>\epsilon$
(by Lemma~\ref{lem:BH-0}), as required.

We now turn to the:

{\em General case.} Fix $C>\epsilon$ and choose any $0<\eta < C- \epsilon$. Since $(f_n)_{n\in \N}$ is uniformly integrable, 
Lemma~\ref{lem:split} gives a sequence $(A_n)_{n\in \N}$ in~$\Sigma$ such that each $g_n:=f_n\chi_{A_n}$
belongs to $L_\infty(\mu,X)$ and:
\begin{itemize}
\item $\sup_{n\in \N}\|g_n\|_{L_\infty(\mu,X)}<\infty$,
\item $\|f_n-g_n\|_{L_1(\mu,X)} \leq \eta$ for every $n\in \N$. 
\end{itemize}
Define a multi-function $\tilde{F}:\Omega \to \mathcal{P}(X)$ by $\tilde{F}(\omega):=F(\omega)\cup\{0\}$
for every $\omega \in \Omega$. By the {\em Particular case}, 
the sequence $(g_n)_{n\in \N}$ in $S_{1}(\tilde{F})$ cannot be an $\ell_1$-sequence 
with constant~$C-\eta$. Since $\|f_n-g_n\|_{L_1(\mu,X)}\leq \eta$
for all $n\in \N$, we conclude that $(f_n)_{n\in \N}$ cannot be an $\ell_1$-sequence with constant~$C$.
The proof is finished.
\end{proof}

\begin{question}\label{question:Bourgain}
Is constant $2$ optimal in Theorem~\ref{theo:Bourgain}?
\end{question}

The concept of $\delta\mathcal{C}$-sets in Lebesgue-Bochner spaces (which corresponds to the case $\epsilon=0$ of the following definition)
appeared first in~\cite{bou-JJ} and was discussed further in~\cite{bat-hie}.

\begin{defi}\label{defi:deltaC} 
Let $(\Omega,\Sigma,\mu)$ be a probability space, $X$ be a Banach space, $W \sub L_1(\mu,X)$, and $\epsilon\geq 0$. 
We say that $W$ is a {\em $\delta\mathcal{C}_\epsilon$-set} if it is uniformly integrable and 
for each $\delta>0$ there is an $\epsilon$-weakly precompact set $M \sub X$ such that
for every $f\in W$ there is $A\in \Sigma$ (depending on~$f$) with $\mu(A) \geq 1-\delta$ 
such that $f(\omega) \in M$ for $\mu$-a.e. $\omega \in A$. If in addition $\epsilon=0$, then we just say that $W$ is a {\em $\delta\mathcal{C}$-set}.
\end{defi}

Every $\delta\mathcal{C}$-set is weakly precompact, but the converse does not hold in general, see~\cite{bou-JJ}
and \cite[Example~3]{bat-hie}. It is known that uniform integrability and being a $\delta\mathcal{C}$-set are equivalent properties if $X$ does not contain
subspaces isomorphic to~$\ell_1$ (see \cite[Corollary~9 and Theorem~14]{bou-JJ}).

The last result of this section is an extension of Corollary~\ref{cor:Bourgain}. To deal with it we need a lemma 
that will also be used later. We omit its straightforward proof.

\begin{lem}\label{lem:sum}
Let $X$ be a Banach space.
\begin{enumerate}
\item[(i)] If $(x_n)_{n\in \N}$ and $(y_n)_{n\in \N}$ are bounded sequences in~$X$, then 
$$
	\delta(x_n+y_n)\leq \delta(x_n)+\delta(y_n).
$$
\item[(ii)] If $\epsilon_1,\epsilon_2\geq 0$ and $M_i \sub X$ is $\epsilon_i$-weakly precompact for each $i\in \{1,2\}$, then $M_1+M_2$ is $(\epsilon_1+\epsilon_2)$-weakly precompact.
\item[(iii)] If $\epsilon\geq 0$ and $M\sub X$ is $\epsilon$-weakly precompact, then $\overline{M}$ is $\epsilon$-weakly precompact. 
\end{enumerate}
\end{lem}

\begin{cor}\label{cor:deltaC}
Let $(\Omega,\Sigma,\mu)$ be a probability space, $X$ be a Banach space, and $\epsilon\geq 0$. Then every $\delta\mathcal{C}_\epsilon$-set
of $L_1(\mu,X)$ is $2\epsilon$-weakly precompact. 
\end{cor}
\begin{proof} Let $W \sub L_1(\mu,X)$ be a $\delta\mathcal{C}_\epsilon$-set. Fix $\eta>0$. Choose $\delta>0$ such that 
$$
	\sup_{f\in W}\int_C \|f(\cdot)\|_X \, d\mu\leq \eta
	\quad
	\text{for every $C\in \Sigma$ with $\mu(C)\leq \delta$.}
$$  
Let $M \sub X$ be an $\epsilon$-weakly precompact set such that, 
for each $f\in W$, there is $A_f\in \Sigma$ such that $\mu(A_f)\geq 1-\delta$ and $f(\omega)\in M$
for $\mu$-a.e. $\omega \in A_f$. Then for each $f\in W$ we have $f\chi_{A_{f}}\in L(M \cup\{0\})$ and $\|f\chi_{\Omega \setminus A_{f}}\|_{L_1(\mu,X)}\leq \eta$, hence 
\begin{equation}\label{eqn:UI}
	W \sub L(M \cup\{0\})+\eta B_{L_1(\mu,X)}. 
\end{equation}
Since $M \cup\{0\}$ is $\epsilon$-weakly precompact, Corollary~\ref{cor:Bourgain} ensures that the set $L(M\cup\{0\})$
is $2\epsilon$-weakly precompact. By~\eqref{eqn:UI} and Lemma~\ref{lem:sum}(ii), $W$~is $(2\epsilon+2\eta)$-weakly precompact.
As $\eta>0$ is arbitrary, $W$ is $2\epsilon$-weakly precompact (Lemma~\ref{lem:epsilon-compact+}).
\end{proof}

\section{Banach spaces having property~$\mathfrak{KM}_w$}\label{section:propertyM}

We begin this section by collecting some basic properties of Banach spaces having property~$\mathfrak{KM}_w$ (Definition~\ref{defi:KMw}).
The first result says somehow that this property can be handled by considering absolutely convex closed subsets of the unit ball.

\begin{pro}\label{pro:characterization}
Let $X$ be a Banach space. The following statements are equivalent:
\begin{enumerate}
\item[(i)] $X$ has property~$\mathfrak{KM}_w$, i.e., there is a family $\{M_{n,p}:n,p\in \Nat\}$ of subsets of~$X$ such that
\begin{enumerate}
\item $M_{n,p}$ is $\frac{1}{p}$-weakly precompact for all $n,p\in \N$;
\item for each weakly precompact set $C \sub X$ and for each $p\in \N$ there is $n\in \N$ such that $C \sub M_{n,p}$.
\end{enumerate}
\item[(ii)] The same as~(i) with each $M_{n,p}$ being absolutely convex and closed.
\item[(iii)] There is a family $\{M_{n,p}:n,p\in \Nat\}$ of subsets of~$B_X$ such that
\begin{enumerate}
\item $M_{n,p}$ is $\frac{1}{p}$-weakly precompact for all $n,p\in \N$;
\item for each weakly precompact set $C \sub B_X$ and for each $p\in \N$ there is $n\in \N$ such that $C \sub M_{n,p}$.
\end{enumerate}
\item[(iv)] The same as (iii) with each $M_{n,p}$ being absolutely convex and closed.
\end{enumerate}
Moreover, any of these families can be chosen such that $M_{n,p} \sub M_{n+1,p}$ for all $n,p\in \N$.
\end{pro}
\begin{proof} The implications (ii)$\impli$(i) and (iv)$\impli$(iii) are obvious.

(i)$\impli$(ii) and (iii)$\impli$(iv): It is easy to check that, for a given $\epsilon\geq 0$, the union of finitely many 
$\epsilon$-weakly precompact subsets of a Banach space is $\epsilon$-weakly precompact. 
So, if $\{M_{n,p}:n,p\in \N\}$ is a family as in condition~(i) (resp.,~(iii)) and we define $\tilde{M}_{n,p}:=\bigcup_{i=1}^n M_{i,p}$ for all $n,p\in \N$, 
then the family $\{\tilde{M}_{n,p}:n,p\in \N\}$ also satisfies the requirements of condition~(i)
(resp.,~(iii)) and we have $\tilde{M}_{n,p}\sub \tilde{M}_{n+1,p}$ for all $n,p\in \N$.
Each ${\rm aco}(\tilde{M}_{n,p})$ is $\frac{2}{p}$-weakly precompact
(by Theorem~\ref{theo:Krein}) and so the same holds for $\overline{{\rm aco}}(\tilde{M}_{n,p})$ 
(Lemma~\ref{lem:sum}(iii)). Thus, the family $\{\overline{{\rm aco}}(\tilde{M}_{n,2p}):n,p\in \N\}$ satisfies the requirements.

(i)$\impli$(iii): Note that if $\{M_{n,p}:n,p\in \N\}$ is a family as in condition~(i), then 
$\{M_{n,p}\cap B_X:n,p\in \N\}$ satisfies the requirements of~(iii).

(iii)$\impli$(i): Let $\phi:\N \to \N \times \N$ be a bijection. Write $\phi(n)=(\phi_1(n),\phi_2(n))$ for all $n\in \N$.
Define $\tilde{M}_{n,p}:=\phi_2(n) M_{\phi_1(n),p \phi_2(n)}$ for all $n,p\in \N$. 
It is easy to check that the family $\{\tilde{M}_{n,p}:n,p\in \N\}$ satisfies the required properties.
\end{proof}
 
\begin{pro}\label{pro:M-subspaces}
Let $X$ be a Banach space having property~$\mathfrak{KM}_w$. Then every subspace of~$X$ has property~$\mathfrak{KM}_w$. 
\end{pro}
\begin{proof} 
Let $\{M_{n,p}:n,p\in \N\}$ be a family of subsets of~$X$ as in Definition~\ref{defi:KMw}.
Then, for any subspace $Y\sub X$, the family $\{M_{n,p}\cap Y:n,p\in \N\}$ witnesses that $Y$
has property~$\mathfrak{KM}_w$. 
\end{proof}

We next show that any SWPG Banach space has property~$\mathfrak{KM}_w$. This includes
all SWCG spaces (e.g., $L_1(\mu)$ for a finite measure~$\mu$, separable Schur spaces and, of course, reflexive spaces)
but also non-SWCG spaces like~$c_0$ (note that every Banach space without isomorphic copies of~$\ell_1$ is SWPG). 
As we will see later, any Banach space having property~$\mathfrak{KM}$ (Definition~\ref{defi:KM}) also 
satisfies property~$\mathfrak{KM}_w$ (Theorem~\ref{theo:wsc}). 

\begin{pro}\label{pro:SWPGimpliesKMw}
Let $X$ be a Banach space. If $X$ is SWPG, then it has property~$\mathfrak{KM}_w$.  
\end{pro}
\begin{proof} Let $C_0 \sub X$ be a weakly precompact set such that, for every weakly precompact set 
$C\sub X$ and for every $\epsilon>0$, there is $n\in \N$ such that $C \sub nC_0+\epsilon B_X$. 
For each $n\in \N$ and for each $p\in \N$ we define
$$
	M_{n,p}:=n C_0 + \frac{1}{2p}B_X,
$$
so that $M_{n,p}$ is $\frac{1}{p}$-weakly precompact (Lemma~\ref{lem:sum}(ii)). Therefore, the family $\{M_{n,p}:n,p\in \N\}$ witnesses that
$X$ has property~$\mathfrak{KM}_w$.
\end{proof}

\begin{question}\label{question:subspacesSWPG}
Let $X$ be a Banach space having property~$\mathfrak{KM}_w$. Is $X$ isomorphic to a subspace
of a SWPG space?
\end{question}

\subsection{Around weak sequential completeness}\label{subsection:wsc}

It is known that every SWCG space is weakly sequentially complete
(see \cite[Theorem~2.5]{sch-whe}). More generally, every Banach space having property~$\mathfrak{KM}$
is weakly sequentially complete (see \cite[Theorem~2.20]{kam-mer2}). On the other hand, it is clear that
if a Banach space is weakly sequentially complete and SWPG, then it is SWCG (cf. \cite[Theorem~2.2]{laj-rod-2}).
We do not know whether properties~$\mathfrak{KM}$ and~$\mathfrak{KM}_w$ are equivalent for weakly sequentially complete spaces
but, as we will see in Corollary~\ref{cor:L-embedded}, this is indeed true for $C$-weakly sequentially complete spaces 
in the sense of~\cite{kal-alt,kal-spu-2}.

The following lemma exploits the argument used to get the weak sequential completeness of SWCG spaces in \cite[Theorem~2.5]{sch-whe}.
A similar idea was used in the proof of the weak sequential completeness of Banach spaces having property~$\mathfrak{KM}$ 
(see \cite[Theorem~2.20]{kam-mer2}) and also in \cite[Lemma~2.3]{laj-rod-2}.

\begin{lem}\label{lem:SW}
Let $X$ be a Banach space and $\{M_{n,p}:n,p\in \Nat\}$ be a family of absolutely convex subsets of~$X$
such that $M_{n,p} \sub M_{n+1,p}$ for all $n,p\in \N$. 
Suppose that for each relatively weakly compact set $L \sub X$ and for each $p\in \N$ there is $n\in \N$ such that $L \sub M_{n,p}$.
Then for each weakly precompact set $H \sub X$ and for each $p\in \N$ there is $n\in \N$ such that $H \sub 2M_{n,p}$.
\end{lem}
\begin{proof} We argue by contradiction. Suppose there exist a weakly precompact set $H \sub X$ and $p\in \N$ such that
$H \not\subseteq 2M_{n,p}$ for all $n\in \Nat$. Let $(x_n)_{n\in \N}$ be a sequence in~$H$
such that $x_n\not\in 2M_{n,p}$ for all $n\in \Nat$. 
Since $H$ is weakly precompact and $M_{n,p}\sub M_{n+1,p}$ for every $n\in \N$, by passing to a subsequence we can assume that $(x_n)_{n\in \N}$ is weakly Cauchy.
Note that the set $\{n\in \Nat: x_n \in 2M_{m,p}\}$ is finite for every $m\in \Nat$.

For each $n\in \Nat$ and for each $s\in \{1,2\}$ we define
$$
	m_s(n):=\min\Big\{m\in \Nat: \, x_n \in sM_{m,p}\Big\},
$$
so that $x_n\in sM_{m,p}$ if and only if $m\geq m_s(n)$. Observe that $m_2(n)\leq m_1(n)$ for all $n\in \N$ (because each $M_{m,p}$ is balanced
and so $M_{m,p}\sub 2M_{m,p}$).
Let $\psi:\Nat \to \Nat$ be any function such that ${\displaystyle \lim_{n\to \infty}}\frac{n}{\psi(n)}=0$
and $n\leq \psi(n)$ for all $n\in \Nat$.

Observe that there is a subsequence $(x_{n_k})_{k\in \N}$ such that
\begin{equation}\label{eqn:m2m1}
	\psi(m_1(n_k)) < m_2(n_{k+1}) \quad \mbox{for all }k\in \Nat.
\end{equation}
Indeed, set $n_1=1$ and suppose that $n_k \in \Nat$ has already been chosen for some $k\in \N$. 
Since $x_n\not \in 2M_{\psi(m_1(n_k)),p}$ whenever $n\geq \psi(m_1(n_k))$, we can choose $n_{k+1}\in \Nat$ with $n_{k+1}>n_k$ such that 
$x_{n_{k+1}}\not\in 2M_{\psi(m_1(n_k)),p}$ 
and so $\psi(m_1(n_k))<m_2(n_{k+1})$. 

On the other hand, since
$$
	m_1(n_k) \leq \psi(m_1(n_k)) \stackrel{\eqref{eqn:m2m1}}{<} m_2(n_{k+1}) \leq m_1(n_{k+1})
	\quad\mbox{for all }k\in \Nat,
$$
the sequence $(m_1(n_k))_{k\in \N}$ is strictly increasing and so
\begin{equation}\label{eqn:limit}
	\lim_{k\to \infty} \frac{m_1(n_k)}{\psi(m_1(n_{k}))}=0.
\end{equation}

Define $y_k:=x_{n_{k+1}}-x_{n_k}$ for all $k\in \Nat$, so that $(y_k)_{k\in\N}$ is a weakly null sequence in~$X$. 
By assumption, there is $m_0\in \Nat$ such that 
$y_k \in M_{m_0,p}$ for all $k\in \N$, hence 
$$
	x_{n_{k+1}}=y_k+x_{n_k}\in M_{m_0,p} + M_{m_1(n_k),p} \sub 2M_{m_1(n_k),p}
$$
whenever $m_1(n_k)\geq m_0$ (bear in mind that $M_{m_1(n_k),p}$ is convex) and so
$$
	\frac{m_1(n_k)}{\psi(m_1(n_k))} \geq \frac{m_2(n_{k+1})}{\psi(m_1(n_k))} \stackrel{\eqref{eqn:m2m1}}{>} 1  
$$
for large enough~$k$. This contradicts~\eqref{eqn:limit} and finishes the proof.
\end{proof}

It turns out that property~$\mathfrak{KM}_w$ can be characterized as follows: 

\begin{pro}\label{pro:defiMp}
Let $X$ be a Banach space. Then $X$ has property~$\mathfrak{KM}_w$ if and only if 
there is a family $\{M_{n,p}:n,p\in \Nat\}$ of subsets of~$X$ such that:
\begin{enumerate}
\item[(i)] $M_{n,p}$ is $\frac{1}{p}$-weakly precompact for all $n,p\in \N$;
\item[(ii)] for each relatively weakly compact set $C \sub X$ and for each $p\in \N$ there is $n\in \N$ such that $C \sub M_{n,p}$.
\end{enumerate}
\end{pro}
\begin{proof} The ``only if'' part is obvious. 
To prove the ``if'' part, note that for each $n\in \N$ and for each $p\in \N$ the set $\bigcup_{i=1}^n M_{i,4p}$ is $\frac{1}{4p}$-weakly precompact
and therefore
$$
	\tilde{M}_{n,p}:={\rm aco}\left(\bigcup_{i=1}^n M_{i,4p}\right)
$$
is $\frac{1}{2p}$-weakly precompact (by Theorem~\ref{theo:Krein}), hence $2\tilde{M}_{n,p}$ is $\frac{1}{p}$-weakly precompact.
Clearly, $\tilde{M}_{n,p} \sub \tilde{M}_{n+1,p}$ for all $n,p\in \N$.
In addition, given any relatively weakly compact set $C \sub X$ and $p\in \N$, there is 
$n\in \N$ such that $C \sub M_{n,4p} \sub \tilde{M}_{n,p}$.
From Lemma~\ref{lem:SW} it follows that $X$ has property~$\mathfrak{KM}_w$ as witnessed by the family $\{2\tilde{M}_{n,p}:n,p\in \Nat\}$.
\end{proof}

\begin{lem}\label{lem:weaklycompact-weaklyprecompact}
Let $X$ be a Banach space, $M \sub X$ and $\epsilon\geq 0$. If $M$ is $\epsilon$-relatively weakly compact,
then $M$ is $2\epsilon$-weakly precompact.
\end{lem}
\begin{proof}
It suffices to check that $M$ is $\epsilon'$-weakly precompact for every $\epsilon'>2\epsilon$ (Lemma~\ref{lem:epsilon-compact+}). 
By contradiction, suppose that $M$ is not $\epsilon'$-weakly precompact. Then $M$ contains an $\ell_1$-sequence 
with constant~$\frac{\epsilon'}{2}$, say~$(x_n)_{n\in \N}$ (by Theorem~\ref{theo:Behrends-KPS}(i)). Hence
$$
	\|x^{**}-x\| \geq \frac{\epsilon'}{2} > \epsilon
	\quad\text{for every $x^{**}\in {\rm clust}_{X^{**}}(x_{n})$ and for every $x\in X$,}
$$
see \cite[Lemma~5(ii)]{kal-alt}. This contradicts that $\overline{M}^{w^*} \sub X+\epsilon B_{X^{**}}$.
\end{proof}

\begin{theo}\label{theo:wsc}
If a Banach space $X$ has property~$\mathfrak{KM}$, then it has property~$\mathfrak{KM}_w$.
\end{theo}
\begin{proof}
Let $\{M_{n,p}:n,p\in \N\}$ be a family of subsets of~$X$ as in Definition~\ref{defi:KM}.
Note that each $M_{n,p}$ is $\frac{2}{p}$-weakly precompact (by Lemma~\ref{lem:weaklycompact-weaklyprecompact}).
Then $\{M_{n,2p}:n,p\in \N\}$ satisfies the conditions of Proposition~\ref{pro:defiMp}, hence
$X$ has property~$\mathfrak{KM}_w$.
\end{proof}

Following \cite{kal-alt,kal-spu-2}, a Banach space~$X$ is said to be {\em $C$-weakly sequentially complete} for some $C\geq 0$ if
for each bounded sequence $(x_n)_{n\in \N}$ in $X$ we have
$$
	\inf_{x\in X}\|x^{**}-x\| \leq C \delta(x_{n}) 
	\quad \text{for every $x^{**}\in {\rm clust}_{X^{**}}(x_n)$.}
$$
This quantitative version of weak sequential completeness appeared 
first in~\cite{god-kal-li}, where it was shown that every $L$-embedded Banach space
is $1$-weakly sequentially complete (see \cite[Lemma~IV.7]{god-kal-li}).
In fact, every $L$-embedded Banach space is $\frac{1}{2}$-weakly sequentially complete,
see \cite[Theorem~1]{kal-alt}. We refer the reader to \cite[Chapter~IV]{har-wer-wer}
for complete information on~$L$-embedded Banach spaces.

\begin{lem}\label{lem:L-embedded}
Let $X$ be a $C$-weakly sequentially complete Banach space for some $C \geq 0$. Let $M \sub X$ and $\epsilon'>\epsilon\geq 0$.
If $M$ is $\epsilon$-weakly precompact, then it is $2C\epsilon'$-relatively weakly compact.
\end{lem}
\begin{proof} There is nothing to prove if $C=0$ (i.e., if $X$ is reflexive), so we assume that $C>0$. 
Suppose that the bounded set $M$ is not $2C\epsilon'$-relatively weakly compact. Then for any $0<\eta< C(\epsilon'-\epsilon)$
there is a sequence $(x_n)_{n\in \N}$ in~$M$ such that 
$$
	C\epsilon' -\eta < \inf \{\|x^{**}-x\|: \, x^{**}\in {\rm clust}_{X^{**}}(x_n), \, x\in X\},
$$
see \cite[Theorem~2.3]{ang-cas2}. Since $X$ is $C$-weakly sequentially complete, 
we deduce that $C\epsilon < C\epsilon'-\eta < C\delta(x_{n_k})$ for every subsequence $(x_{n_k})_{k\in \N}$.
Hence $M$ is not $\epsilon$-weakly precompact.
\end{proof}

\begin{cor}\label{cor:L-embedded}
Let $X$ be a $C$-weakly sequentially complete Banach space for some $C\geq 0$. 
Then $X$ has property~$\mathfrak{KM}$ if and only if it has property~$\mathfrak{KM}_w$.
\end{cor}
\begin{proof} 
The ``only if'' part follows from Theorem~\ref{theo:wsc} and does not require the additional assumption on~$X$.
Conversely, suppose that $X$ has property $\mathfrak{KM}_w$ and let $\{M_{n,p}:n,p\in \N\}$ be a family of subsets of~$X$ as in Definition~\ref{defi:KMw}. Choose
$q\in \N$ with $q > 2C$. Observe that
$M_{n,qp}$ is $\frac{1}{p}$-relatively weakly compact for all $n,p\in \N$, by Lemma~\ref{lem:L-embedded}. 
Therefore, $\{M_{n,qp}:n,p\in\N\}$ satisfies the conditions of Definition~\ref{defi:KM}
and $X$ has property~$\mathfrak{KM}$.
\end{proof}

\begin{question}\label{question:wsc}
Let $X$ be a weakly sequentially complete Banach space having property~$\mathfrak{KM}_w$. Does
$X$ have property~$\mathfrak{KM}$?
\end{question}

\subsection{Quotients}\label{subsection:quotients}

In general, property~$\mathfrak{KM}_w$ is not preserved by quotients. Indeed, every separable Banach space is a quotient of~$\ell_1$
(which is SWCG and so it has property~$\mathfrak{KM}_w$), but there are separable Banach spaces without 
property~$\mathfrak{KM}_w$, like $C[0,1]$ (see Corollary~\ref{cor:CK} in the next subsection). The main result
of this subsection is a three-space type result for property~$\mathfrak{KM}_w$, namely:

\begin{theo}\label{theo:quotient-Mp}
Let $X$ be a Banach space and $Y \sub X$ be a subspace not containing subspaces isomorphic to~$\ell_1$.
Then $X$ has property~$\mathfrak{KM}_w$ if and only if $X/Y$ has property~$\mathfrak{KM}_w$.
\end{theo}

To deal with Theorem~\ref{theo:quotient-Mp} we need a quantitative version of Mazur's theorem
on the existence of norm convergent convex block subsequences of weakly convergent sequences. The following result
is a particular case of \cite[Theorem~4.1]{ang-cas-nam}:

\begin{theo}\label{theo:ACN}
Let $Z$ be a Banach space, $(z_j)_{n\in \N}$ be a bounded sequence in~$Z$, $K \sub Z^*$ be a $w^*$-compact set, and $a>0$. Suppose that
for each $z^*\in K$ there is $j_{z^*}\in \N$ such that 
$|z^*(z_j)|\leq a$ for every $j\geq j_{z^*}$.
Then for every $\epsilon>0$ there is $z_\epsilon\in {\rm co}(\{z_j:j\in \N\})$ such that $|z^*(z_\epsilon)| \leq a +\epsilon$
for all $z^*\in K$. 
\end{theo}

\begin{cor}\label{cor:ACN}
Let $Z$ be a Banach space, $(z_j)_{j\in \N}$ be a bounded sequence in~$Z$, $(K_n)_{n\in \N}$ be a sequence of $w^*$-compact subsets of~$Z^*$, 
and $a_n,\eta_n>0$ for all $n\in \N$. Suppose that
for each $n\in \N$ and for each $z^*\in K_n$ there is $j_{n,z^*}\in \N$ such that 
$|z^*(z_j)|\leq a_n$ for every $j\geq j_{n,z^*}$.
Then there is a convex block subsequence 
$(y_n)_{n\in \N}$ of~$(z_j)_{j\in \N}$ such that 
$$
	|z^{*}(y_n)| \leq a_n+\eta_n
	\quad
	\text{for every $z^{*}\in K_n$ and for every $n\in \N$.}
$$
\end{cor}
\begin{proof} We will construct the $y_n$'s inductively. 
For the first step, apply Theorem~\ref{theo:ACN} (with $K=K_1$) to get
$y_1\in {\rm co}(\{z_j:j\in \N\})$ such that $|z^*(y_1)| \leq a_1+\eta_1$ for every $z^{*}\in K_1$.
Write $y_1=\sum_{j\in J_1}a_j z_j$ for some
finite set $J_1 \sub \N$ and some collection of non-negative real numbers $(a_j)_{j\in J_1}$ with $\sum_{j\in J_1}a_j=1$.  
Pick $j_1\in \N$ with $j_1>\max J_1$. Now, Theorem~\ref{theo:ACN} applied to the tail $(z_j)_{j \geq j_1}$
and $K=K_2$ ensures the existence of $y_2\in {\rm co}(\{z_j:j \geq j_1\})$ such that 
$|z^*(y_2)| \leq a_2+\eta_2$ for every $z^{*}\in K_2$. By continuing in this way 
we get the required convex block subsequence $(y_n)_{n\in \N}$ of~$(z_j)_{j\in \N}$.
\end{proof}

Corollary~\ref{cor:ACN} will be used in the proof of the following quantitative version 
of Lohman's lifting~\cite{loh} (cf. \cite[2.4.a]{cas-gon}) as well as in Subsection~\ref{subsection:propertyK}.

\begin{pro}\label{pro:Lohman}
Let $X$ be a Banach space and $Y \sub X$ be a subspace not containing subspaces isomorphic to~$\ell_1$.
Let $q:X\to X/Y$ be the quotient operator, $M \sub X$ be a bounded set and $\epsilon\geq 0$. If $q(M)$ is $\epsilon$-weakly precompact,
then $M$ is $\epsilon$-weakly precompact. 
\end{pro}
\begin{proof}
Fix $\epsilon'>\epsilon$ (to apply Lemma~\ref{lem:epsilon-compact+}) and suppose that $M$ is not $\epsilon'$-weakly precompact. 
By Theorem~\ref{theo:Behrends-KPS}(i), there is an $\ell_1$-sequence with constant $\frac{\epsilon'}{2}$ contained in~$M$, 
say $(x_n)_{n\in \N}$. Since $q(M)$ is $\epsilon$-weakly precompact, by passing to a subsequence we can assume that $\delta(q(x_n))\leq \epsilon$. Define $z_m:=x_{2m+1}-x_{2m}$ for all $m\in \N$, so that
$(z_m)_{m\in \N}$ is an $\ell_1$-sequence with constant~$\epsilon'$. 

Fix any $0<\eta<\epsilon'-\epsilon$. Then
$$
	\sup_{\varphi\in B_{(X/Y)^*}}\limsup_{m\to \infty}|\varphi(q(z_{m}))| \leq \delta(q(x_{n})) < \epsilon+\eta.
$$
Hence we can apply Corollary~\ref{cor:ACN} to $(q(z_m))_{m\in \N}$ (with $K_n=B_{(X/Y)^*}$ for all $n\in \N$) 
to get a convex block subsequence $(\tilde{z}_k)_{k\in \N}$ of~$(z_m)_{m\in \N}$
such that 
$$
	\|q(\tilde{z}_k)\|_{X/Y}<\epsilon+\eta
	\quad \text{for every $k\in \N$.} 
$$	

Since $(z_m)_{m\in \N}$ is an $\ell_1$-sequence with constant~$\epsilon'$,
the same holds for $(\tilde{z}_k)_{k\in \N}$ (as it can be easily checked).
Choose a sequence $(y_k)_{k\in \N}$ in~$Y$ such that $\|\tilde{z}_k-y_k\|<\epsilon+\eta$ for all $k\in \N$. 
Then $(y_k)_{k\in \N}$ is an $\ell_1$-sequence (with constant~$\epsilon'-\epsilon-\eta)$,
which contradicts the fact that $Y$ contains no subspace isomorphic to~$\ell_1$.
\end{proof}

\begin{proof}[Proof of Theorem~\ref{theo:quotient-Mp}] Let $q:X\to X/Y$ be the quotient operator.
Suppose first that $X$ has property~$\mathfrak{KM}_w$ and let $\{M_{n,p}:n,p\in \Nat\}$ 
be a family of subsets of~$X$ satisfying the conditions of Definition~\ref{defi:KMw}.
By Lemma~\ref{lem:operator-compact}, each $q(M_{n,p})$ is $\frac{1}{p}$-weakly precompact.
Let $C \sub X/Y$ be a weakly precompact set. Choose a bounded set $L \sub X$ such that $q(L)=C$.
Then $L$ is weakly precompact by Lohman's lifting (i.e., Proposition~\ref{pro:Lohman} with $\epsilon=0$). Therefore, 
for every $p\in \N$ there is $n\in \N$ such that $L \sub M_{n,p}$,
and so $C \sub q(M_{n,p})$. This shows that $X/Y$ has property~$\mathfrak{KM}_w$, as witnessed
by the family $\{q(M_{n,p}):n,p\in \Nat\}$.

Conversely, suppose $X/Y$ has property~$\mathfrak{KM}_w$ and let $\{\tilde{M}_{n,p}:n,p\in \Nat\}$ 
be a family of subsets of~$X/Y$ satisfying the conditions of Definition~\ref{defi:KMw}. 
We can assume that $\tilde{M}_{n,p} \sub \tilde{M}_{n+1,p}$ for all $n,p\in \N$ (Proposition~\ref{pro:characterization}).
Take bounded sets $M_{n,p}\sub X$ in such a way that $q(M_{n,p})=\tilde{M}_{n,p}$
and $M_{n,p} \sub M_{n+1,p}$ for all $n,p\in \N$. By Proposition~\ref{pro:Lohman},
each $M_{n,p}$ is $\frac{1}{p}$-weakly precompact and so 
$$
	M'_{n,p}:=M_{n,p}+nB_Y \sub X
$$ 
is $\frac{1}{p}$-weakly precompact (Lemma~\ref{lem:sum}(ii)).

We claim that the family $\{M'_{n,p}:n,p\in \N\}$ witnesses that $X$ has property~$\mathfrak{KM}_w$. Indeed,
let $L \sub X$ be a weakly precompact set and fix $p\in \N$. Since $q(L)$ is weakly precompact in~$X/Y$, 
we have $q(L) \sub \tilde{M}_{n_0,p}=q(M_{n_0,p})$ for some $n_0\in \N$.
Now, if we choose $n\in \N$ large enough such that $n\geq n_0$ and $n\geq \|x\|_X$ for all $x\in L-M_{n_0,p}$,
then we have $L \sub M_{n,p} + n B_Y=M'_{n,p}$. 
\end{proof}

A similar argument using Lohman's lifting
and the fact that weak precompactness is preserved by taking absolutely convex hulls
(Subsection~\ref{subsection:aco}) yields the following result. We omit the details.

\begin{theo}\label{theo:quotientSWPG}
Let $X$ be a Banach space and $Y \sub X$ be a subspace not containing subspaces isomorphic to~$\ell_1$.
Then $X$ is SWPG if and only if $X/Y$ is SWPG.
\end{theo}

\begin{question}\label{question:3space}
Let $X$ be a Banach space and $Y \sub X$ be a subspace.
\begin{enumerate}
\item[(i)] Does $X$ have property~$\mathfrak{KM}_w$ if both $Y$ and $X/Y$ have property~$\mathfrak{KM}_w$?
\item[(ii)] Is $X$ SWPG if both $Y$ and $X/Y$ are SWPG?
\end{enumerate}
\end{question}
 
We stress that it is also unknown whether the property of being SWCG is a three-space property.  
There is a result analogous to Theorem~\ref{theo:quotientSWPG} for the property of being SWCG when $Y$ is reflexive,
see \cite[Theorem~2.7]{sch-whe}.

\subsection{Unconditional sums} \label{subsection:sums}

In this subsection we discuss the stability of property~$\mathfrak{KM}_w$ under unconditional sums.

Given a sequence $(X_m)_{m\in \N}$ of Banach spaces, we denote by $\left(\bigoplus_{m\in \N}X_m\right)_{\ell_\infty}$
its $\ell_\infty$-sum. If $E$ is a Banach space with a normalized $1$-unconditional basis~$(e_m)_{m\in \N}$,
we write $\left(\bigoplus_{m\in \N}X_m\right)_E$ to denote the $E$-sum of $(X_m)_{m\in \N}$, that is, the Banach space
$$
	\left(\bigoplus_{m\in \N}X_m\right)_E:=
	\left\{
	(x_m)_{m\in \N}\in \prod_{m\in \N}X_m: \, \sum_{m\in \N}\|x_m\| \, e_m \text{ converges in $E$} 
	\right\}
$$
equipped with the norm
$$
	\left\|(x_m)_{m\in \N}\right\|_{(\bigoplus_{m\in \N}X_m)_E}
	:=
	\left\|\sum_{m\in \N}\|x_m\|\, e_m\right\|_E.
$$
The simplest examples are obtained when $E$ is~$c_0$ or $\ell_p$ for $1\leq p<\infty$ (equipped with its usual basis).

The fact that the $\ell_1$-sum of countably many SWPG spaces is SWPG 
was stated without proof in \cite[Example~4.1(c)]{kun-sch}. 
We include a proof below (Proposition~\ref{pro:l1-sums-SWPG}) for the sake of completeness and also because a similar argument yields the analogue
for Banach spaces having property~$\mathfrak{KM}_w$ (Proposition~\ref{pro:l1-sums}). The key is condition~($\Delta$) of the following folk lemma. 

\begin{lem}\label{lem:truncation}
Let $(X_m)_{m\in \N}$ be a sequence of Banach spaces and let us write $X:=(\bigoplus_{m\in \N}X_m)_{\ell_1}$. 
Let $\pi_m:X\to X_m$ be the canonical projection for all $m\in \N$.
Then a set $C \sub X$ is weakly precompact if and only if $\pi_m(C)$ is weakly precompact in~$X_m$ for all $m\in \N$ and
the following condition holds: 
\begin{enumerate}
\item[($\Delta$)] for every $\delta>0$ there is $m_0\in \N$ such that
$$	
	\sum_{m>m_0} \|\pi_m(x)\|_{X_m} \leq \delta
	\quad
	\text{for all $x\in C$.}
$$
\end{enumerate}
\end{lem}
\begin{proof} The ``if'' part follows easily from Lemma~\ref{lem:sum}(ii) and the fact 
that $X^*$ can be identified with $\left(\bigoplus_{m\in \N}X^*_m\right)_{\ell_\infty}$,
the duality being
$$
	\big\langle (x_m)_{m\in \N},(x_m^*)_{m\in \N}\big\rangle=
	\sum_{m\in \N}\langle x_m,x^*_m\rangle
$$
for every $(x_m)_{m\in \N}\in X$ and for every $(x^*_m)_{m\in \N} \in X^*$.

To prove the ``only if'' part, suppose that $C \sub X$ is weakly precompact. 
Clearly, $\pi_m(C)$ is weakly precompact in~$X_m$ for all $m\in \N$. 
For each $n\in \N$ and for each $p\in \N$, we define
$$
	M_{n,p}:=\{x\in X: \, \pi_m(x)=0 \text{ for all $m>n$}\}+\frac{1}{2p}B_X.
$$
Since~($\Delta$) holds for any relatively weakly compact subset of~$X$ (see, e.g., \cite[Lemma~7.2(ii)]{kac-alt}), 
the family $\{M_{n,p}:n,p\in \N\}$ satisfies the conditions of Lemma~\ref{lem:SW}. Fix $\delta>0$ and choose $p\in \N$ such that 
$\frac{1}{p}\leq \delta$. By Lemma~\ref{lem:SW}, there is $m_0\in \N$ such that $C \sub 2M_{m_0,p}$,
which implies that $\sum_{m>m_0} \|\pi_m(x)\|_{X_m} \leq \delta$ for all $x\in C$.
\end{proof}

\begin{pro}\label{pro:l1-sums-SWPG}
Let $(X_m)_{m\in \N}$ be a sequence of SWPG Banach spaces. Then 
$(\bigoplus_{m\in \N}X_m)_{\ell_1}$ is SWPG.
\end{pro}
\begin{proof} Write $X:=(\bigoplus_{m\in \N}X_m)_{\ell_1}$.
For each $m\in \N$ we fix a weakly precompact set $C_m \sub \frac{1}{2^m}B_{X_m}$ with the following property: for every weakly precompact
set $H \sub X_m$ and for every $\epsilon>0$ there is $n\in \N$ such that $H \sub nC_m+\epsilon B_{X_m}$. 
Since weak precompactness is preserved by taking absolutely convex hulls (see Subsection~\ref{subsection:aco}),
we can assume that each $C_m$ is absolutely convex.
Define
$$
	W:= \prod_{m\in \N}C_m \sub X,
$$
so that $W$ is weakly precompact in~$X$ (by Lemma~\ref{lem:truncation}). Now, take any weakly precompact set $C \sub X$
and $\epsilon>0$. Lemma~\ref{lem:truncation} applied to~$C$ ensures the existence of $m_0\in \N$ such that
\begin{equation}\label{eqn:truncation1}	
	\sum_{m>m_0} \|\pi_m(x)\|_{X_m} \leq \frac{\epsilon}{2}
	\quad
	\text{for all $x\in C$.}
\end{equation}
For each $m\in \{1,\dots,m_0\}$ the set $\pi_m(C) \sub X_m$ is weakly precompact and so there is $n_m\in \N$ such that
$\pi_m(C) \sub n_m C_m + \frac{\epsilon}{2m_0}B_{X_m}$. If we write $n:=\max\{n_1,\dots,n_{m_0}\}$, then 
\begin{equation}\label{eqn:coordinate}
	\pi_m(C) \sub nC_m + \frac{\epsilon}{2m_0}B_{X_m}
	\quad\text{for every $m\in \{1,\dots,m_0\}$}
\end{equation}
(because each $C_m$ is balanced). From~\eqref{eqn:truncation1} and~\eqref{eqn:coordinate} we get $C \sub nW+\epsilon B_X$.
This shows that $X$ is SWPG. 
\end{proof}

\begin{pro}\label{pro:l1-sums}
Let $(X_m)_{m\in \N}$ be a sequence of Banach spaces having property~$\mathfrak{KM}_w$. Then 
$(\bigoplus_{m\in \N}X_m)_{\ell_1}$ has property~$\mathfrak{KM}_w$.
\end{pro}
\begin{proof} Write $X:=(\bigoplus_{m\in \N}X_m)_{\ell_1}$.
Fix $m\in \N$. Let us take a family $\{M^m_{n,p}:n,p\in \N\}$ of subsets of~$X_m$ satisfying the conditions of Definition~\ref{defi:KMw},
with the additional property that $M^m_{n,p} \sub M^m_{n+1,p}$ for all $n,p\in \N$ (Proposition~\ref{pro:characterization}).
If we denote by $j_m:X_m\to X$ the canonical embedding, then each $\tilde{M}^m_{n,p}:=j_m(M^m_{n,p})$ is $\frac{1}{p}$-weakly precompact in~$X$
(apply Lemma~\ref{lem:operator-compact}). 

Let $\phi:\N \to \N \times \N$ be a bijection and write $\phi(n)=(\phi_1(n),\phi_2(n))$ for all $n\in \N$.
For each $n\in \N$ and for each $p\in\N$, the set
$$
	M_{n,p}:=\sum_{m=1}^{\phi_2(n)} \tilde{M}^m_{\phi_1(n),2p\phi_2(n)}+\frac{1}{4p}B_X
$$
is $\frac{1}{p}$-weakly precompact in~$X$ (apply Lemma~\ref{lem:sum}(ii)). 
Fix a weakly precompact set $C \sub X$ and take any $p\in \N$. By Lemma~\ref{lem:truncation},
there is $m_0\in \N$ such that 
\begin{equation}\label{eqn:truncation}
	\sum_{m>m_0} \|\pi_m(x)\|_{X_m} \leq \frac{1}{4p}
	\quad
	\text{for all $x\in C$,}
\end{equation}
where $\pi_m:X\to X_m$ is the canonical projection. 
For each $m\in\{1,\dots,m_0\}$ the set $\pi_m(C)$ is weakly precompact,
hence $\pi_m(C) \sub M^m_{k_m,2pm_0}$ for some $k_m\in \N$. Write $k:=\max\{k_1,\dots,k_{m_0}\}$,
so that $\pi_m(C) \sub M^m_{k,2pm_0}$ for every $m\in\{1,\dots,m_0\}$.
Take $n\in \N$ such that $\phi(n)=(\phi_1(n),\phi_2(n))=(k,m_0)$. Then
$$
	C \stackrel{\eqref{eqn:truncation}}{\sub} \sum_{m=1}^{m_0} j_m(\pi_m(C)) + \frac{1}{4p}B_X 
	\sub \sum_{m=1}^{m_0} \tilde{M}^m_{k,2pm_0} + \frac{1}{4p}B_X = M_{n,p}.
$$
This shows that the family $\{M_{n,p}:n,p\in \N\}$ fulfills the conditions of Definition~\ref{defi:KMw}
and so $X$ has property~$\mathfrak{KM}_w$. 
\end{proof}

The situation changes for $\ell_p$-sums when $1<p<\infty$. Indeed, if $(X_m)_{m\in\N}$ is a sequence of Banach spaces
and $(\bigoplus_{m\in \N}X_m)_{\ell_p}$ is SWPG (or just a subspace of a SWPG space), then $X_m$ contains 
no subspace isomorphic to~$\ell_1$ for all but finitely many $m\in \N$, see \cite[Theorem~4.5]{kun-sch}
(resp., \cite[Theorem~2.6]{laj-rod-2}). In particular, for $1<p<\infty$ the space $\ell_p(\ell_1)$ does not 
embed isomorphically into a SWPG space. Theorem~\ref{theo:E-sequence} below uses similar ideas to extend
those results to property~$\mathfrak{KM}_w$ and more general unconditional sums.

The following well-known lemma provides a useful representation for the dual of an unconditional sum.

\begin{lem}\label{lem:E-dual}
Let $(X_m)_{m\in \N}$ be a sequence of Banach spaces, $E$ be a Banach space with a normalized $1$-unconditional basis~$(e_m)_{m\in \N}$, 
and let us write $X:=\left(\bigoplus_{m\in \N}X_m\right)_{E}$ to denote the corresponding $E$-sum. Let
$(e_m^*)_{m\in\N}$ be the sequence in~$E^*$ of biorthogonal functionals
associated with~$(e_m)_{m\in \N}$. Then:
\begin{enumerate}
\item[(i)] For every ${\bf x^*}=(x_m^*)_{m\in \N}\in \prod_{m\in \N}X_m^*$ satisfying
$$
	|||{\bf x^*}|||:=\sup_{M\in \N}\left\|\sum_{m=1}^M \|x_m^*\| \, e_m^*\right\|_{E^*}<\infty
$$
we can define $\varphi_{{\bf x^*}} \in X^*$ by 
$$
	\varphi_{{\bf x^*}}((x_m)_{m\in \N}):=\sum_{m\in\N}x_m^*(x_m)
	\quad
	\text{ for all $(x_m)_{m\in \N}\in X$.}
$$	
Moreover, $\|\varphi_{{\bf x^*}}\|_{X^*}=|||{\bf x^*}|||$.
\item[(ii)] For every $\varphi\in X^*$ there is ${\bf x^*}\in \prod_{m\in \N}X_m^*$ with $|||{\bf x^*}|||<\infty$
such that $\varphi=\varphi_{{\bf x^*}}$.
\end{enumerate}
\end{lem}

\begin{lem}\label{lem:E-weaklynull}
Let $(X_m)_{m\in \N}$ be a sequence of Banach spaces and $E$ be a Banach space with a normalized $1$-unconditional basis~$(e_m)_{m\in \N}$ such that $E^*$ is separable.
Let $X:=\left(\bigoplus_{m\in \N}X_m\right)_{E}$ be the corresponding $E$-sum and $\pi_m:X\to X_m$ be the canonical projection for all $m\in \N$.
Then a sequence $(y_j)_{j\in \N}$ in~$X$ is weakly null if and only if it is bounded and 
the sequence $(\pi_m(y_{j}))_{j\in \N}$ is weakly null in~$X_m$ for every $m\in \N$.
\end{lem}
\begin{proof} The ``only if'' part is immediate and does not require the assumption that $E^*$ is separable.

Let us prove the ``if'' part. We follow the notations of Lemma~\ref{lem:E-dual}.
Fix $\varphi\in X^*$ and write $\varphi=\varphi_{{\bf x^*}}$ for some ${\bf x^*}=(x^*_m)_{m\in \N}\in \prod_{m\in \N}X_m^*$
with $|||{\bf x^*}|||<\infty$. Since $E^*$ is separable, 
$(e_m^*)_{m\in\N}$ is a normalized $1$-unconditional boundedly-complete basis of~$E^*$
(see, e.g., \cite[Theorems~3.2.12 and 3.3.1]{alb-kal}), hence the series 
$\sum_{m\in \N} \|x_m^*\| \, e_m^*$ converges unconditionally in~$E^*$. Take any $\epsilon>0$. Choose $M\in \N$
large enough such that $\left\|\sum_{m\in F}\|x_m^*\| \, e_m^*\right\|_{E^*}\leq \frac{\epsilon}{2C}$
for every finite set $F \sub \N\setminus \{1,\dots,M\}$, where $C>0$ is a constant
such that $\|y_j\|_X\leq C$ for all $j\in \N$. Consider the elements
$$
	{\bf x_0^*}:=(x_1^*,\dots,x_M^*,0,0,\dots)
	\quad\text{and}\quad
	{\bf x_1^*}:=(\underbrace{0,\dots,0}_{\text{$M$ times}},x_{M+1}^*,x_{M+2}^*,\dots)
$$ 
of $\prod_{m\in \N}X_m^*$. Then $|||{\bf x_1^*}|||\leq \frac{\epsilon}{2C}$ and $\varphi=\varphi_{{\bf x_0^*}}+ \varphi_{{\bf x_1^*}}$, so
there is $j_0\in \N$ such that
$$
	|\varphi(y_j)| \leq |\varphi_{{\bf x_0^*}}(y_j)| + |\varphi_{{\bf x_1^*}}(y_j)|
	\leq \sum_{m=1}^M \big| x_m^*\big(\pi_m(y_j)\big)\big| + \frac{\epsilon}{2} \leq \epsilon
$$
for every $j\geq j_0$ (each $(\pi_m(y_j))_{j\in \N}$ is weakly null in~$X_m$).
It follows that $(\varphi(y_j))_{j\in \N}$ converges to~$0$. As $\varphi\in X^*$ is arbitrary, $(y_j)_{j\in \N}$ is 
weakly null.
\end{proof}

\begin{theo}\label{theo:E-sequence}
Let $(X_m)_{m\in \N}$ be a sequence of Banach spaces and $E$ be a Banach space with a 
normalized $1$-unconditional basis such that $E^*$ is separable.
If the $E$-sum $X:=(\bigoplus_{m\in \N}X_m)_E$ has property~$\mathfrak{KM}_{w}$, then $X_m$ contains no subspace isomorphic to~$\ell_1$ for all but finitely many $m\in \N$.
\end{theo}
\begin{proof} Since property~$\mathfrak{KM}_w$ in inherited by subspaces (Proposition~\ref{pro:M-subspaces}), it suffices to prove that 
if $X_m$ contains a subspace isomorphic to~$\ell_1$ for every $m\in \N$, then $X$ fails property~$\mathfrak{KM}_{w}$.
So, we assume that each $X_m$ contains a subspace isomorphic to~$\ell_1$.
By James' $\ell_1$-distortion theorem (see, e.g., \cite[Theorem~10.3.1]{alb-kal}), for each $m\in \N$
there is a normalized $\ell_1$-sequence with constant~$\frac{1}{2}$ in~$X_m$, say~$(x_k^m)_{k\in \N}$. 
Suppose, by contradiction, that $X$ has property~$\mathfrak{KM}_{w}$.
Fix a family $\{M_{n,p}:n,p\in \Nat\}$ of subsets of~$X$ as in Definition~\ref{defi:KMw}.

Let $\Phi \sub \Nat^\Nat$ be the set of all strictly increasing sequences in~$\N$. 
Fix $\vf\in \Phi$. For each $j\in \Nat$, define  
$$
	y_{\vf,j}:=(0,\dots,0,\underbrace{x_{\vf(j)}^j}_{j\text{th-term}},0,\dots) \in X.
$$
The sequence $(y_{\vf,j})_{j\in \N}$ is weakly null in~$X$ (by Lemma~\ref{lem:E-weaklynull}). 
Therefore, the set $K_{\vf}:=\{y_{\vf,j}:j\in \Nat\}$ is relatively weakly compact in~$X$.
Hence $K_\vf\sub M_{n(\vf),2}$ for some $n(\vf)\in \Nat$. 

We claim that there is $n_0\in \Nat$ such that 
$$
	A_{n_0}:=\{\vf(n_0):\, \vf\in \Phi,\, n(\vf)=n_0\}
$$ 
is infinite. Indeed,
otherwise we can find $\vf_0\in \Phi$ such that 
$$
	\vf_0(n)>\max\{\vf(n):\, \vf\in \Phi,\, n(\vf)=n\}
	\quad\text{for all $n\in \N$,}
$$
which leads to a contradiction when $n=n(\vf_0)$. 

Enumerate $A_{n_0}=\{\vf_k(n_0):k\in \N\}$ for some sequence $(\vf_k)_{k\in \N}$ in~$\Phi$ with $n(\vf_k)=n_0$ for all $k\in \N$.
Define $z_k:=y_{\varphi_k,n_0}\in K_{\vf_k} \sub M_{n_0,2}$ for all $k\in \Nat$. Observe that each $z_k$ has norm~$1$ and that
for every $s\in \N$ and for all $a_1,\dots,a_s\in \mathbb R$ we have
$$
	\left\|\sum_{k=1}^s a_k z_k\right\|_{X}=\left\|\sum_{k=1}^s a_k x_{\vf_k(n_0)}^{n_0}\right\|_{X_{n_0}}
	\geq \frac{1}{2} \sum_{k=1}^s |a_k|.
$$
It follows that $M_{n_0,2}$ contains an $\ell_1$-sequence with constant $\frac{1}{2}$, which
contradicts the fact that $M_{n_0,2}$ is $\frac{1}{2}$-weakly precompact
(Theorem~\ref{theo:Behrends-KPS}(ii)). The proof is finished.
\end{proof}

The following corollary extends the corresponding result for $\ell_p$-sums and $1<p<\infty$, which was
proved in~\cite[Corollary~2.28]{kam-mer2}.

\begin{cor}\label{cor:E-sequence}
Let $(X_m)_{m\in \N}$ be a sequence of Banach spaces and $E$ be a Banach space with a 
normalized $1$-unconditional basis such that $E^*$ is separable.
If the $E$-sum $(\bigoplus_{m\in \N}X_m)_E$ has property~$\mathfrak{KM}$, then $X_m$ is reflexive
for all but finitely many $m\in \N$.
\end{cor}
\begin{proof} On the one hand, since property~$\mathfrak{KM}$ implies weak sequential completeness (see \cite[Theorem~2.20]{kam-mer2}),
each $X_m$ is weakly sequentially complete. On the other hand, by Theorems~\ref{theo:wsc} and~\ref{theo:E-sequence}, 
$X_m$ contains no subspace isomorphic to~$\ell_1$ for all but finitely many $m\in \N$. Now,
the conclusion follows from the fact that every weakly sequentially complete Banach space not containing subspaces isomorphic to~$\ell_1$ is reflexive.
\end{proof}

\begin{cor}\label{cor:l2l1}
The spaces $c_0(\ell_1)$ and $\ell_p(\ell_1)$ for $1<p<\infty$ fail property~$\mathfrak{KM}_w$. 
\end{cor}

The space $C[0,1]$ contains an isometric copy of any separable Banach space, so from the previous corollary and 
Proposition~\ref{pro:M-subspaces} we get:

\begin{cor}\label{cor:CK}
The space $C[0,1]$ fails property~$\mathfrak{KM}_w$.
\end{cor}

The Banach space of Batt and Hiermeyer~\cite[\S3]{bat-hie} (which we will denote by~$X_{BH}$)
was the first example of a weakly sequentially complete space which is not SWCG (see~\cite[Example~2.6]{sch-whe}). 
Simpler examples like $\ell_2(\ell_1)$ were given later (see \cite[Theorem~5.10]{sch-whe-2}).
It is known that $X_{BH}$ also fails property~$\mathfrak{KM}$
(see \cite[Example~2.9]{mer-sta-2} and \cite[Theorem~2.18]{kam-mer2}). We will show that, in fact,
$X_{BH}$ fails property~$\mathfrak{KM}_w$, because it contains 
an isometric copy of~$\ell_2(\ell_1)$ (Corollary~\ref{cor:embedding-l2l1}).

The space $X_{BH}$ can be seen as a member of a class of Banach spaces built on adequate families which
goes back to Kutzarova and Troyanski~\cite{kut-tro-2} (see, e.g., \cite{arg-mer}). Recall
that a family $\cA$ of subsets of a non-empty set~$\Gamma$ is said to be {\em adequate} if
it satisfies the following conditions:
\begin{enumerate}
\item[(i)] If $A\in \cA$ and $B \sub A$, then $B\in \mathcal{A}$. 
\item[(ii)] $\{\gamma\}\in \cA$ for all $\gamma\in \Gamma$.
\item[(iii)] If $A \sub \Gamma$ and every finite subset of~$A$ belongs to~$\mathcal{A}$, then $A\in \cA$.  
\end{enumerate}
For each function $x:\Gamma\to \erre$, we define
$$
	\|x\|_{E_{1,2}(\cA)}:=
	\sup 
	\left(
	\sum_{i=1}^p 
	\left(
	\sum_{\gamma \in C_i}
	|x(\gamma)|
	\right)^2
	\right)^{1/2}\in [0,\infty],
$$
where the supremum runs over all finite collections $C_1,\dots,C_p$ of pairwise disjoint finite elements of~$\mathcal{A}$.
The Banach space $E_{1,2}(\mathcal{A})$ is 
$$
	E_{1,2}(\mathcal{A}):=\big\{x\in \erre^\Gamma: \, \|x\|_{E_{1,2}(\cA)}<\infty\big\},
$$
equipped with the pointwise operations and the norm $\|\cdot\|_{E_{1,2}(\cA)}$.
Clearly, $E_{1,2}(\mathcal{A})$ contains the linear space $c_{00}(\Gamma)$ of all finitely supported real-valued functions on~$\Gamma$.
For each $\gamma\in \Gamma$, let $e_\gamma\in c_{00}(\Gamma)$ be defined by $e_\gamma(\gamma'):=0$
for all $\gamma'\neq \gamma$, and $e_\gamma(\gamma):=1$.

\begin{pro}\label{pro:embedding-l2l1}
Let $\mathcal{A}$ be an adequate family of subsets of a non-empty set~$\Gamma$. Suppose there is a sequence $(A_n)_{n\in \N}$ in~$\mathcal{A}$ 
such that each $A_n$ is infinite and 
$$
	|\{n\in \N: \, A_n\cap A\neq \emptyset\}|\leq 1 
	\quad \text{for every $A\in \mathcal{A}$.}
$$
Then $\ell_2(\ell_1)$ is isometrically isomorphic to a subspace of $E_{1,2}(\mathcal{A})$. 
\end{pro}
\begin{proof} Observe that the $A_n$'s are pairwise disjoint.
Choose a countable infinite set $\{\gamma_{n,m}:n,m\in \N\} \sub \Gamma$ in such a way that $A_n \supseteq \{\gamma_{n,m}:m\in \N\}$ for all $n\in \N$. 
Fix $N,M\in \N$ and $a_{n,m}\in \mathbb{R}$ for all $1\leq n \leq M$ and for all $1\leq m \leq M$. We will prove that
\begin{equation}\label{eqn:Aviles}
	\left\|\sum_{n=1}^N\sum_{m=1}^M a_{n,m} e_{\gamma_{n,m}} \right\|_{E_{1,2}(\mathcal{A})} =
	\left(\sum_{n=1}^N \left(\sum_{m=1}^M |a_{n,m}|\right)^2\right)^{1/2}.
\end{equation}  
Clearly, this implies that the space $\ell_2(\ell_1)$ is isometrically isomorphic to
the subspace $\overline{{\rm span}}\{e_{\gamma_{n,m}}:n,m\in \N\} \sub E_{1,2}(\cA)$.  

Write $x:=\sum_{n=1}^N\sum_{m=1}^M a_{n,m} e_{\gamma_{n,m}}$ and 
$$
	\Gamma_0:=\{\gamma_{n,m}:\, 1\leq n\leq N, \, 1\leq m \leq M\}.
$$
Define $\tilde{A}_n:=\{\gamma_{n,m}:1\leq m\leq M\}$ for every $1\leq n \leq N$.
Observe that each $\tilde{A}_n$ belongs to~$\mathcal{A}$ (because $\tilde{A}_n \sub A_n\in \cA$) 
and $\tilde{A}_n\cap \tilde{A}_{n'}=\emptyset$ whenever $n\neq n'$, so
$$
	\left\|x \right\|_{E_{1,2}(\mathcal{A})} 
	\geq
	\left(
	\sum_{n=1}^N 
	\left(
	\sum_{\gamma \in \tilde{A}_n} |x(\gamma)|
	\right)^2
	\right)^{1/2}
	=
	\left(
	\sum_{n=1}^N 
	\left(
	\sum_{m=1}^M |a_{n,m}|
	\right)^2
	\right)^{1/2}.
$$	

On the other hand, fix any finite collection $C_1,\dots,C_p$ of pairwise disjoint finite elements of~$\mathcal{A}$. 
For each $1\leq n\leq N$, we write 
$$
	I_n:=\{1\leq i \leq p: \, \tilde{A}_n \cap C_i \neq \emptyset\}.
$$   
Observe that the $I_n$'s are pairwise disjoint and that $\Gamma_0\cap C_i=\tilde{A}_n \cap C_i$
for every $i\in I_n$ and for every $1\leq n \leq N$. Then
\begin{multline*}
	\sum_{i=1}^p \left(\sum_{\gamma\in C_i} |x(\gamma)|\right)^2=
	\sum_{i=1}^p \left(\sum_{\gamma\in \Gamma_0 \cap C_i} |x(\gamma)|\right)^2 \\
	= \sum_{n=1}^N \sum_{i\in I_n} \left(\sum_{\gamma\in \tilde{A}_n \cap C_i} |x(\gamma)|\right)^2
	\leq  \sum_{n=1}^N \left(\sum_{i\in I_n} \sum_{\gamma\in \tilde{A}_n \cap C_i} |x(\gamma)|\right)^2 \\
	 = \sum_{n=1}^N \left(\sum_{\gamma \in \tilde{A}_n \cap (\bigcup_{i\in I_n} C_i)} |x(\gamma)|\right)^2
	 \leq \sum_{n=1}^N \left(\sum_{m=1}^M |a_{n,m}|\right)^2.
\end{multline*}
By the very definition of the norm in~$E_{1,2}(\mathcal{A})$, it follows that 
$$
	\|x\|_{E_{1,2}(\mathcal{A})} \leq \left(\sum_{n=1}^N \left(\sum_{m=1}^M |a_{n,m}|\right)^2\right)^{1/2}.
$$
This finishes the proof of~\eqref{eqn:Aviles}. 
\end{proof}

The Batt-Hiermeyer space~\cite[\S 3]{bat-hie} is defined as $X_{BH}:=E_{1,2}(\cA)$
where $\cA$ is the adequate family of all chains of the {\em dyadic tree}~$\mathcal{T}$, i.e.,
the set $\cT:=\{\emptyset\}\cup \bigcup_{n\in \Nat}\{0,1\}^n$ of all finite sequences of $0$'s and $1$'s.
By a {\em chain} of~$\cT$ we mean a set $A \sub \mathcal{T}$ such that for every $\sigma,\sigma' \in A$
we have that either $\sigma$ extends~$\sigma'$ or vice versa. The space $X_{BH}$ is a separable, weakly sequentially complete, dual Banach space.

\begin{cor}\label{cor:embedding-l2l1}
The space $X_{BH}$ contains a subspace isometrically isomorphic to~$\ell_2(\ell_1)$. 
In particular, $X_{BH}$ fails property~$\mathfrak{KM}_w$.
\end{cor}
\begin{proof}
The last assertion follows from the first one, Corollary~\ref{cor:l2l1} and Proposition~\ref{pro:M-subspaces}.
In order to prove the first assertion it is enough to check that the family of all chains of~$\mathcal{T}$
satisfies the condition of Proposition~\ref{pro:embedding-l2l1}. 
For each $n\in \N$ and for each $m\in \N$ we define
$$
	\sigma_{n,m}:=(\underbrace{0,0,\dots,0}_{\text{$n$ times}},\underbrace{1,1,\dots,1}_{\text{$m$ times}})\in \mathcal{T}.
$$
Now, it is easy to check that the sequence of chains $(A_n)_{n\in\N}$ of~$\mathcal{T}$ defined by
$$
	A_n:=\{\sigma_{n,m}: \, m\in \N\}
$$ 
satisfies the required condition.
\end{proof}

\subsection{Property~$\mathfrak{KM}_w$ and Lebesgue-Bochner spaces}\label{subsection:strongLB}

Given a probability space $(\Omega,\Sigma,\mu)$ and a Banach space~$X$, it is unknown whether the property
of being SWPG lifts from~$X$ to~$L_1(\mu,X)$. This is indeed the case if $X$ contains no subspace isomorphic to~$\ell_1$
(see \cite[Example~3.5(ii)]{laj-rod-2}). As it was pointed out in \cite[Remark~3.4]{laj-rod-2}, a general positive
answer would imply that the property of being SWCG lifts from~$X$ to~$L_1(\mu,X)$, thus answering a long standing open question of Schl\"{u}chtermann and
Wheeler~\cite{sch-whe}. 

In this subsection we address the same type of question for property~$\mathfrak{KM}_w$. Similarly, we do not know
whether $L_1(\mu,X)$ has property~$\mathfrak{KM}_w$ whenever $X$ does. 
The following partial answer is similar to previous results for SWCG and SWPG spaces, see \cite[Theorem~2.7]{rod13} and \cite[Proposition~3.8]{laj-rod-2}. 
It involves the notion of $\delta\mathcal{C}$-set which was recalled in Definition~\ref{defi:deltaC}.

\begin{pro}\label{pro:Lebesgue-BochnerSG}
Let $(\Omega,\Sigma,\mu)$ be a probability space and $X$ be a Banach space having property~$\mathfrak{KM}_w$. Then there is a family
$\{\tilde{M}_{n,p}:n,p\in \N\}$ of subsets of~$L_1(\mu,X)$ such that:
\begin{enumerate}
\item[(i)] $\tilde{M}_{n,p}$ is $\frac{1}{p}$-weakly precompact for all $n,p\in \N$.
\item[(ii)] For each $\delta\mathcal{C}$-set $W \sub L_1(\mu,X)$ and for each $p\in \N$ there is $n\in \N$ such that $W \sub \tilde{M}_{n,p}$.
\end{enumerate}
\end{pro}
\begin{proof}
Let $\{M_{n,p}:n,p\in \N\}$ be a family of subsets of~$X$ as in Definition~\ref{defi:KMw}.
For each $n\in \N$ and for each $p\in \N$, the set
$$
	\tilde{M}_{n,p}:=L(M_{n,4p})+\frac{1}{4p}B_{L_1(\mu,X)}
$$
is $\frac{1}{p}$-weakly precompact in~$L_1(\mu,X)$, by Corollary~\ref{cor:Bourgain} and Lemma~\ref{lem:sum}(ii).
To check condition~(ii), fix a $\delta\mathcal{C}$-set $W \sub L_1(\mu,X)$ and take any $p\in \N$.
As in the proof of Corollary~\ref{cor:deltaC} (with $\epsilon=0$), there is a weakly precompact set $C \sub X$ such that
$W \sub L(C) + \frac{1}{4p}B_{L_1(\mu,X)}$. Since $C \sub M_{n,4p}$ for some $n\in \N$, we get $W \sub \tilde{M}_{n,p}$.
\end{proof}

Our next result strengthens the conclusion of the Bourgain-Maurey-Pisier theorem 
(i.e., Theorem~\ref{theo:Bourgain} with $\epsilon=0$) for Banach spaces having property~$\mathfrak{KM}_w$ and multi-functions which are ``measurable'' in a certain sense.  
The proof is similar to that of \cite[Proposition~2.8]{rod17}.

\begin{pro}\label{pro:LR}
Let $(\Omega,\Sigma,\mu)$ be a probability space, $X$ be a Banach space having property~$\mathfrak{KM}_w$ and
$F:\Omega \to \mathcal{P}(X)$ be a multi-function such that:
\begin{enumerate}
\item[(i)] $F(\omega)$ is weakly precompact  for $\mu$-a.e. $\omega\in \Omega$;
\item[(ii)] $\{\omega \in \Omega: \ F(\omega)\sub C\}\in \Sigma$ for every absolutely convex closed set $C \sub X$.
\end{enumerate}
Write
$$
	S_{1}(F):=\{f\in L_1(\mu,X): \, f(\omega)\in F(\omega) \text{ for $\mu$-a.e. $\omega\in \Omega$}\}
$$
to denote the set of all (equivalence classes of) Bochner $\mu$-integrable selectors of~$F$. Then
every uniformly integrable subset of $S_{1}(F)$ is a $\delta\mathcal{C}$-set.
\end{pro}
\begin{proof} It suffices to check that for every $\delta>0$ there exist 
a weakly precompact set $W \sub X$ and $A\in \Sigma$ with $\mu(A)\geq 1-\delta$ such that $F(\omega)\sub W$ for every $\omega\in A$.

Let $\{M_{n,p}:n,p\in \N\}$ be a family of absolutely convex closed subsets of~$X$ witnessing property~$\mathfrak{KM}_w$
and such that $M_{n,p}\sub M_{n+1,p}$ for all $n,p\in \N$ (Proposition~\ref{pro:characterization}). Then
$$
	A_{n,p}:=\{\omega \in \Omega: \, F(\omega) \sub M_{n,p}\} \in \Sigma
$$
for all $n,p\in \N$. Given any $p\in \Nat$, we have $\mu(\bigcup_{n\in \Nat}A_{n,p})=1$ (because $F(\omega)$
is weakly precompact for $\mu$-a.e. $\omega \in \Omega$) and $A_{n,p} \sub A_{n+1,p}$ for every $n\in \N$, 
so there is $n_p\in \Nat$ such that $\mu(A_{n_p,p})\geq 1-\frac{\delta}{2^{p}}$. 

Define $A:=\bigcap_{p\in \Nat}A_{n_p,p}\in \Sigma$, so that $\mu(A) \geq 1-\delta$. 
Observe that
$$
	F(\omega)\sub W:=\bigcap_{p\in \Nat} M_{n_p,p}
	\quad\text{for every $\omega\in A$}
$$  
and that $W$ is weakly precompact, because each $M_{n_p,p}$ is $\frac{1}{p}$-weakly precompact
(and then we can apply Lemma~\ref{lem:epsilon-compact+}).
\end{proof}

\subsection{A remark on property~(K)}\label{subsection:propertyK}

A Banach space $X$ is said to have {\em property~(K)} if every $w^*$-convergent sequence in~$X^*$
admits a convex block subsequence which converges with respect to the Mackey topology $\mu(X^*,X)$, that is,
the topology on~$X^*$ of uniform convergence on weakly compact subsets of~$X$.
This concept is due to Kwapie\'{n} and appeared first in~\cite{kal-pel}. 
Property~(K) and some related properties have also been studied in \cite{avi-rod,che-alt,dep-dod-suk,fig-alt2,fra-ple,rod19}.

Subspaces of SWCG spaces have property~(K) (see \cite[Corollary~2.3]{avi-rod}, cf. \cite[Corollary~3.8]{dep-dod-suk}).
The main result of this subsection generalizes that statement:

\begin{theo}\label{theo:KMimpliesK}
If a Banach space~$X$ has property~$\mathfrak{KM}$, then it has property~(K).
\end{theo}

In the proof of Theorem~\ref{theo:KMimpliesK} we will use the following
quantitative extension of \cite[Lemma~2.11]{avi-rod}:

\begin{lem}\label{lem:epsilon-K}
Let $X$ be a Banach space, $(x_j^*)_{j\in \N}$ be a $w^*$-null sequence in~$B_{X^*}$, and $\epsilon_n>\epsilon \geq 0$ for all $n\in \N$. 
If $(M_n)_{n\in \N}$ is a sequence of $\epsilon$-relatively weakly compact subsets of~$X$, then there is a convex block subsequence 
$(y_n^*)_{n\in \N}$ of~$(x_j^*)_{j\in \N}$ such that 
$$
	|y_n^*(x^{**})| \leq \epsilon_n
	\quad
	\text{for every $x^{**}\in \overline{M_n}^{w^*}$ and for every $n\in \N$.}
$$
\end{lem}
\begin{proof} Fix $n\in \N$ and choose $\eta_n>0$ such that $\epsilon+2\eta_n \leq \epsilon_n$.
Observe that for every $x^{**}\in \overline{M_n}^{w^*}$ there is 
$x\in X$ such that $\|x^{**}-x\|\leq \epsilon$, hence 
$$
	|x_j^*(x^{**})|\leq \epsilon+|x_j^*(x)| \leq \epsilon + \eta_n
$$ 
for $j\in \N$ large enough (depending on~$x^{**}$). 

So, we can apply Corollary~\ref{cor:ACN} to the sequence $(x^*_j)_{j\in \N}$ (with $Z=X^*$, $K_n=\overline{M_n}^{w^*}$ and $a_n=\epsilon+\eta_n$ for all $n\in \N$)
to obtain the required convex block subsequence.
\end{proof}

\begin{proof}[Proof of Theorem~\ref{theo:KMimpliesK}]
Let $\{M_{n,p}:n,p\in \N\}$ be a family of subsets of~$X$ satisfying the conditions of Definition~\ref{defi:KM}. Observe that, for any $\epsilon\geq 0$, the union of finitely many $\epsilon$-relatively weakly compact subsets of a Banach space is $\epsilon$-relatively weakly compact. So, we can assume without loss
of generality that $M_{n,p}\sub M_{n+1,p}$ for every $n\in \N$ and for every $p\in \N$.

Fix a $w^*$-null sequence in~$B_{X^*}$ (to check property~(K) it suffices to consider such sequences).
We can apply inductively Lemma~\ref{lem:epsilon-K} to get, for each $p\in \N$, a sequence $(x^*_{n,p})_{n\in \N}$ in such a way that:
\begin{itemize}
\item $(x^*_{n,1})_{n\in \N}$ is a convex block subsequence of $(x^*_{n})_{n\in \N}$;
\item $(x^*_{n,p+1})_{n\in \N}$ is a convex block subsequence of $(x^*_{n,p})_{n\in \N}$ for all $p\in\N$;
\item $\sup\{|x_{n,p}^*(x^{**})|:\, x^{**}\in \overline{M_{n,p}}^{w^*}\} \leq \frac{1}{p}+\frac{1}{n}$ for all $n,p\in \N$. 
\end{itemize}
This inequality implies that, for each $n' \geq n$ in~$\N$ and for each $p\in \N$, we have
\begin{equation}\label{eqn:K}
	\sup_{x^{**}\in \overline{M_{n,p}}^{w^*}} |x^*(x^{**})| \leq \frac{1}{p}+\frac{1}{n'}
	\quad
	\text{for every $x^*\in {\rm co}(\{x_{m,p}^*: m \geq n'\})$.}
\end{equation}

Define $\tilde{x}^*_k:=x^*_{k,k}$ for every $k\in \N$. Then $(\tilde{x}^*_k)_{k\in \N}$ is 
a convex block subsequence of $(x^*_n)_{n\in \N}$.
We claim that $(\tilde{x}^*_k)_{k\in \N}$ is $\mu(X^*,X)$-convergent to~$0$. Indeed, let $C \sub X$ be a weakly compact set and fix $\epsilon>0$.
Choose $p\in \N$ with $\frac{1}{p}\leq \epsilon$ and then take $n_0\in \N$ such that $C \sub M_{n_0,p}$. Since $(\tilde{x}^*_k)_{k\geq p}$
is a convex block subsequence of $(x^*_{n,p})_{n\in \N}$, inequality~\eqref{eqn:K} yields
$$
	\limsup_{k\to\infty} \, \sup_{x\in C} |\tilde{x}_k^*(x)| \leq \frac{1}{p} \leq \epsilon.
$$
As $\epsilon>0$ is arbitrary, it follows that $(\tilde{x}^*_k)_{k\in \N}$ converges to~$0$ uniformly on~$C$.
\end{proof}

In \cite[Theorem~2.18]{kam-mer2} it was pointed out that every Banach space~$X$ having 
property~$\mathfrak{KM}$ is a subspace of a weakly compactly generated space, 
as a consequence of \cite[Theorem~1]{fab-mon-ziz-2}; therefore, $B_{X^*}$ is $w^*$-sequentially compact (see, e.g., \cite[p.~228, Theorem~4]{die-J}).
This fact and Theorem~\ref{theo:KMimpliesK} give the following result.

\begin{cor}\label{cor:seq-compact}
Let $X$ be a Banach space having property~$\mathfrak{KM}$. Then every bounded sequence in~$X^*$
admits a $\mu(X^*,X)$-convergent convex block subsequence.
\end{cor}

In general, property~$\mathfrak{KM}_w$ does not imply property~(K). For instance, it is easy to see
that $c_0$ fails property~(K) (see, e.g., \cite[p.~4998]{avi-rod}). 

A result of \O rno and Valdivia (see, e.g., \cite[Theorem~5.53]{fab-ultimo}) states that every $\mu(X^*,X)$-convergent sequence in~$X^*$ is norm convergent
whenever the Banach space~$X$ contains no subspace isomorphic to~$\ell_1$ (and conversely). We finish the paper with
an application of that result which, in particular, shows that property~(K) can be defined via
uniform convergence on weakly precompact sets.

\begin{pro}\label{pro:Mackey-sequential}
Let $X$ be a Banach space and $(x_n^*)_{n\in \N}$ be a sequence in~$X^*$ which converges to some~$x^*\in X^*$ with respect to~$\mu(X^*,X)$.
Then $(x_n^*)_{n\in \N}$ converges to~$x^*$ uniformly on each weakly precompact subset of~$X$.
\end{pro}
\begin{proof} Of course, we can assume that $x^*=0$. Since the absolutely convex hull of a weakly precompact set is 
weakly precompact (see Subsection~\ref{subsection:aco}), it suffices to check that $(x_n^*)_{n\in \N}$ converges to~$0$ uniformly on each absolutely convex 
weakly precompact set $M \sub X$. To this end, we apply the Davis-Figiel-Johnson-Pe\l cz\'{y}nski factorization method
to~$M$ (see, e.g., \cite[p.~250, Lemma~8]{die-uhl-J}) to get a Banach space $Y$ 
and an operator~$T:Y\to X$ such that $M \sub T(B_Y)$. Since $M$ is weakly precompact, $Y$ contains no subspace
isomorphic to~$\ell_1$ (see, e.g., \cite[Corollary~1.5]{nei}).
Since $T^*:X^*\to Y^*$ is $\mu(X^*,X)$-to-$\mu(Y^*,Y)$ continuous, the sequence $(T^*(x_n^*))_{n\in\N}$
converges to~$0$ with respect to $\mu(Y^*,Y)$. By the aforementioned result of \O rno and Valdivia, 
$\|T^*(x_n^*)\|_{Y^*}\to 0$ as $n\to \infty$. Since
$$
	\sup_{x\in M}|x_n^*(x)| \leq \sup_{y\in B_Y}\big|x_n^*(T(y))\big|=\|T^*(x_n^*)\|_{Y^*}
	\quad\mbox{for all }n\in \N,
$$
we conclude that $(x_n^*)_{n\in \N}$ converges to~$0$ uniformly on~$M$.
\end{proof}

\subsection*{Acknowledgements}
The author thanks A.~Avil\'{e}s and V.~Kadets for valuable comments
on some parts of Subsection~\ref{subsection:sums}.
The research is partially supported by {\em Agencia Estatal de Investigaci\'{o}n} [MTM2017-86182-P, grant cofunded by ERDF, EU] 
and {\em Fundaci\'on S\'eneca} [20797/PI/18].


\bibliographystyle{amsplain}

\end{document}